\documentclass[a4paper,11pt]{amsart}

\usepackage[top=1.65in, bottom=1.65in, right=1.3in, left=1.3in]{geometry}

\usepackage{color}
\usepackage{xcolor}

\usepackage{amssymb}
\usepackage{amsthm}
\usepackage{graphics}
\usepackage{amsmath}
\usepackage{amstext}

\usepackage{amsfonts}
\usepackage{hyperref}
\usepackage{amscd}
\usepackage[utf8]{inputenc} 
\usepackage[T1]{fontenc} 
\usepackage{fancyhdr}

\usepackage{xfrac}
\usepackage{braket}
\usepackage{mathtools}
\usepackage{stmaryrd}
\usepackage{mathrsfs}
\usepackage{extarrows}

\usepackage{microtype}

\usepackage{graphicx}
\usepackage{caption}
\usepackage{subcaption}

\usepackage{todonotes}
\usepackage{import}

\usepackage{ esint }
\usepackage{abstract}

\usepackage{enumitem,linegoal}
\usepackage{calc}

\usepackage{tikz}
\usetikzlibrary{matrix,arrows,calc,patterns,shapes,trees,positioning}

\title[Hyperbolicity and Bifurcations for skew products]{Hyperbolicity and Bifurcations in 
holomorphic
families of polynomial skew products}

 \thanks{This research was partially supported by the ANR project LAMBDA, ANR-13-BS01-0002, and by a PEPS grant
 Jeune-s Chercheur-e-s
 awarded to the first author}

\begin{author}[M.~Astorg]{Matthieu Astorg}
\address{ 
Universit\'e d'Orl\'eans, MAPMO \\
 UMR CNRS 7349\\
45067 Orl\'eans Cedex 2\\
  France }
  \email{matthieu.astorg$@$univ-orleans.fr}
\end{author}

\begin{author}[F.~Bianchi]{Fabrizio Bianchi}
\address{ 
CNRS, Univ. Lille, UMR 8524 - Laboratoire Paul Painlev\'e, F-59000 Lille, France}
  \email{fabrizio.bianchi$@$univ-lille.fr}
\end{author}

\input pdfcolor

\newcommand{\Dd}{{\mathcal D}}

\newcommand{\Pp}{{\mathcal P}}

\newcommand{\Rr}{{\mathcal R}}

\def\eps{\varepsilon}

\def\im{\mathrm{Im}}
\def\id{\mathrm{Id}}

\def\lam{\lambda}

\def\supp{\mathrm{supp}}

\newcommand{\mb}{\mathbb}

\newcommand{\R}{\mb R}
\newcommand{\C}{\mb C}
\newcommand{\N}{\mb N}
\newcommand{\D}{\mb D}
\newcommand{\Z}{\mb Z}

\renewcommand{\P}{\mb P}

\renewcommand{\phi}{\varphi}
\renewcommand{\epsilon}{\varepsilon}
\renewcommand{\bar}{\overline}
\renewcommand{\tilde}{\widetilde}

\def\({\left(}
\def\){\right)}

\def\b#1{\bar{#1}}
\def\t#1{\tilde{#1}}

\newcommand{\pa}[1]{\left(#1\right)}

\renewcommand{\bra}[1]{\left[#1\right]}
\newcommand{\abs}[1]{\left|#1\right|}
\newcommand{\norm}[1]{\left\|#1\right\|}

\newtheorem{teo*}{Theorem}

\newtheorem{teo}{Theorem}[section]
\newtheorem{defi}[teo]{Definition}
\newtheorem{cor}[teo]{Corollary}
\newtheorem{lemma}[teo]{Lemma}
\newtheorem{prop}[teo]{Proposition}
\newtheorem{remark}[teo]{Remark}

\newtheorem{claim}[teo]{Claim}

\newtheorem{teointroletter}{Theorem}

\DeclareMathOperator{\Supp}{Supp}

\newcommand{\htbif}{\hat T_\mathrm{bif}}

\DeclareMathOperator{\Bif}{Bif}

\DeclareMathOperator{\Per}{Per}

\newcommand{\limn}{\lim_{n \rightarrow \infty}}
\renewcommand{\id}{\mathrm{Id}}

\newcommand{\mcal}{\mathcal{M}}

\newcommand{\bcal}{\mathcal{B}}
\newcommand{\ptwo}{\mathbb{P}^2}
\newcommand{\pk}{\mathbb{P}^k}

\newcommand{\card}{\mathrm{card\,}}

\newcommand{\jac}{\mathrm{Jac}}

\newcommand{\skp}{\mathbf{Sk}(p,2)}

\newcommand{\skpa}{\mathbf{Sk}(p,2,\alpha)}

\newcommand{\tbif}{T_\mathrm{bif}}
\newcommand{\tbifz}{T_{\mathrm{bif},z}}
\newcommand{\tbifzc}{T_{\mathrm{bif},z,c}}
\newcommand{\per}{\mathrm{Per}}

\renewcommand{\im}{\mathrm{Im}}
\newcommand{\re}{\mathrm{Re}}
\newcommand{\lloc}{L_{\mathrm{loc}}^1}
\newcommand{\ccal}{\mathcal{C}}
\newcommand{\dcal}{\mathcal{D}}

\newtheorem{lem}[teo]{Lemma}

\begin{document}

\maketitle

\begin{abstract}
We initiate a parametric study of
holomorphic
 families of polynomial skew products, i.e.,
polynomial endomorphisms
of $\C^2$ of the form $F(z,w)= \pa{p(z), q(z,w)}$
that extend to holomorphic 
endomorphisms of $\P^2(\C)$.
We prove that stability in the sense of \cite{bbd2015} preserves hyperbolicity within such families, 
and give a complete classification
of the hyperbolic 
components that are the analogue, in this setting, of the complement of the Mandelbrot set
for the family $z^2 +c$. We also precisely describe
the geometry of the bifurcation locus and current near the boundary of the parameter space.
One of our tools is an asymptotic
 equidistribution property
for
the bifurcation current. This is established in the
general setting of families of endomorphisms of $\P^k$, and is the first equidistribution
result of this kind for holomorphic dynamical systems in dimension larger than one.
\end{abstract}

\setcounter{secnumdepth}{3}
\setcounter{tocdepth}{1}
\tableofcontents

\section{Introduction and results}

A polynomial skew product in two complex variables is a polynomial endomorphism
of $\C^2$ of the form $F(z,w)= \pa{p(z), q(z,w)}$
that extends to an endomorphism of $\P^2=\P^2 (\C)$.
The dynamics of these maps was studied in detail in \cite{jonsson1999dynamics}.
Despite (and actually because of) their specific form,
they
have already provided examples
of dynamical phenomena
not displayed by one-dimensional polynomials,
see for instance \cite{astorg2014two,dujardin2016nonlaminar,dujardin2016non,taflin_blender}.
Their understanding has
  already proved
to be a necessary step in the study of endomorphisms of $\P^k$ in any dimension.
In this paper we address the question of understanding the dynamical stability of such maps. In order to do this,
let us first
introduce the framework for our work.

A \emph{holomorphic family} of endomorphisms of $\P^k$
is a holomorphic map $f\colon M\times \P^k \to M\times \P^k$
of the form $f(\lam,z)= (\lam, f_\lam(z))$. The complex manifold $M$ is the \emph{parameter
space}
and we require that all $f_\lam$ have the same  degree.
In dimension $k=1$,
the study of \emph{stability} and \emph{bifurcation} within such families
was initiated
by Man\'e-Sad-Sullivan \cite{mane1983dynamics} and Lyubich \cite{lyubich1983some} in the 80's. They proved that many natural
definitions
of stability are equivalent, allowing one to decompose 
the parameter space of any family of rational maps into a \emph{stability locus}
and a \emph{bifurcation locus}.
Moreover, their 
notion of stability   preserves hyperbolicity, in the
sense that if a parameter in a given component of the stability locus is hyperbolic, then all parameters
in the same component enjoy the same property. This fact
is of crucial importance 
for the theory.
In 2000, by means of the \emph{Lyapunov function} $L(f_\lam)$,
DeMarco \cite{demarco2001dynamics} constructed  a natural
\emph{bifurcation current} $\tbif := dd^c_{\lam} L(f_\lam)$ precisely supported on the bifurcation
locus. This allowed for the start
of a pluripotential study of the bifurcations of rational maps.

The theory by Man\'e-Sad-Sullivan, Lyubich and DeMarco
was recently extended to any dimension by Berteloot, Dupont, and the second author \cite{bbd2015,b_misiurewicz}.
Despite the quite precise understanding of the relation
between the various phenomena related to stability
and bifurcation (motion of the repelling cycles, Lyapunov function, Misiurewicz parameters),
apart from specific examples (\cite{bt_desboves})
or near special parameters (\cite{bb_hausdorff,dujardin2016non,taflin_blender,biebler2019lattes}),
we still miss a concrete and somehow general family
whose bifurcations
can be explicitly exhibited and studied, which may possibly play the role of the quadratic family
$z^2 +c$
for the higher dimensional theory.
 Moreover, it
is an open question
whether stability
preserves hyperbolicity in this context.

This paper aims at the precise understanding of the phenomena above, hyperbolicity \emph{in primis},
within families of polynomial skew products.

\medskip

\subsection{Main results}

While many of the results apply to more general families, we mainly
focus here on the family of \emph{quadratic skew products}, i.e.,
skew products of (algebraic) degree 2, that are in this context the analogue 
of the family  $z^2 +c$.
It is not difficult to see 
 that the
dynamical study of this family
can be reduced to that of the family
\begin{equation}\label{eq_family_intro}
\tag{*}
f_\lam: (z,w)\mapsto (z^2 + d, w^2 + az^2 + bz + c)
\end{equation}
with $d$ and $\lam:=(a,b,c)$ as (complex) parameters. Since bifurcations due to the parameter $d$
are of one-dimensional nature, we fix here $p(z):=z^2 +d$
and consider the parameter space $\skp:=\{f_\lam:  (a,b,c)\in \C^3\}$.

We are especially interested
in parameters near the boundary of this space, i.e., near the hyperplane
at infinity, that
we denote by $\P^2_\infty$.
The following is our first main result, giving a complete
description
of the bifurcation locus near $\P^2_\infty$
from both a topological and measure-theoretical point of view.
We denote by $J_p$  
the Julia set of $p$.
Given $z\in \C$, we set  $E_z := \set{[a,b,c]\colon a z^2 + bz +c=0}\subset\P^2_\infty$ 
and $E:= \cup_{z\in J_p} E_z$. 
An analogous result for quadratic rational maps is proved in \cite{berteloot2015geometry}.

\begin{teointroletter}\label{teo_new_bif}
The accumulation on $\P^2_\infty$
of the bifurcation locus of the family \eqref{eq_family_intro}
coincides with $E$. Moreover, the 
 bifurcation current $\tbif$ on $\C^3$ extends as a positive closed current $\htbif$ to
  $\P^3 = \C^3 \cup \P^2_\infty$
and
\[
\htbif \wedge [\P^2_{\infty}] = \int_{z\in J_p} [E_z]\mu_p.
\]
\end{teointroletter}

The proof of this result relies on several ingredients. The first is a decomposition for the bifurcation
current (and locus), valid in all the parameter space (and actually for any family of polynomial skew products), see Theorem
\ref{teo_new_skew_things}.

We
 then 
prove that special dynamically defined hypersurfaces
$\Per_n^v (\eta)$ \emph{equidistribute} towards the bifurcation current
$\tbif$
(and $\htbif$), 
see the
 next Section \ref{sec_intro_equi} for more details. Moreover, we can precisely
control the intersections of these
hypersurfaces 
with
 the hyperplane at infinity. We  thus obtain the convergences
\[
\frac{1}{d^{2n}}[\Per_n^v (\eta)] \to \htbif \quad \mbox{ and } \quad 
\frac{1}{d^{2n}}[\Per_n^v (\eta)] \wedge [\P^2_\infty] \to  \int_{z\in J_p} [E_z]\mu_p.
\]
Theorem \ref{teo_new_bif}
then reduces to proving that the
convergences above
imply that
\[
\frac{1}{d^{2n}}[\Per_n^v (\eta)] \wedge [\P^2_\infty] \to \htbif \wedge [\P^2_\infty],
\]
which is a problem of intersection of currents.
To do this, we 
 exploit 
the theory of \emph{horizontal positive closed currents} as developed by Dujardin \cite{dujardin2004henon},
 see also 
 \cite{dinh2006geometry}. This requires proving
  some uniform estimates on the directions at which the bifurcation locus approaches $\P^2_\infty$.

\medskip

Once the bifurcation locus near the hyperplane 
at infinity
is understood, we turn our attention to its complement, and in particular to the characterization of 
the \emph{hyperbolic components}.
Notice that, in order for those to exist,  $p$ must be
hyperbolic.

The stability of a polynomial skew product
as in \eqref{eq_family_intro} is determined by
the behaviour
of the critical points of the form $(z,0)$
with $z\in J_p$. For instance, as is the case for polynomials, when all these points
escape to infinity
by iteration, the map is hyperbolic. It is however not clear a priori 
that the presence of a hyperbolic map in a component
forces all the other maps in the same stability component to be hyperbolic.

In our next result not only do we solve this general
problem in the setting of polynomial skew products
(thus giving meaning to the expression \emph{hyperbolic components} here), 
but we also
give a complete classification of hyperbolic components that are analoguous to the 
so-called \emph{shift locus} from dimension 1.

More precisely, let $\Dd$ be the set of parameters for which all critical points in $J_p \times \C$
escape, and let $\Dd' \subset \Dd$ be the subset of parameters $\lam$ for which there is an arc 
joining $\lam$ to $\ptwo_\infty \backslash E$ inside $\Dd$.
Set
\[\mathcal{S}_p:=\Big\{s : \pi_0(\mathring{K_p}) \to \N : \sum_{U \in \pi_0(\mathring{K_p})} s(U) \leq 2 \Big\},\]
where $\pi_0(\mathring{K_p})$
denotes the set of bounded Fatou components of $p$.

\begin{teointroletter}\label{teo_new_hyp}
	Let $(f_\lam)_{\lam \in M}$ be a holomorphic family of polynomial skew products.
\begin{enumerate}
\item Any $f_\lam$ in a stable component containing a hyperbolic parameter is hyperbolic.
\item Assume that $M=\skp$. All connected components of $\Dd'$ are hyperbolic components, and there is a natural bijection between $\mathcal{S}_p$ and 
the connected components of $\Dd'$.
\end{enumerate}
\end{teointroletter}

The condition of the base polynomial
$p$
of being 
hyperbolic is actually not necessary, if we replace hyperbolicity with
\emph{vertical expansion}, see \cite{jonsson1999dynamics}
and Section \ref{s:vertical}. 
Our Theorem holds in this case too (see Section \ref{section_vertical_hyp}), 
and proves that stability preserves vertical expansion (Theorem \ref{teo_stab_hyp}), and
gives  a classification of vertical expanding component (Theorem \ref{teo_classification_components}).

The proof of the first item of Theorem  \ref{teo_new_hyp}
is based on a characterization of hyperbolicity (and vertical expansion)
due to Jonsson (see Theorem \ref{teo_jonsson_unique})
based on the (non) accumulation of the postcritical set
on the Julia set. 
Our task is to prove that stability preserves this equivalent notion.
The proof of the second item 
is topological in nature. Our main task is to exclude that a given hyperbolic component can accumulate two
distinct components of $\ptwo_\infty \backslash E$. 
 To prove this, we show that 
the combinatorial invariants
 $s \in \mathcal{S}_p$
  encode the isotopy class of the Julia set in $J_p \times \C$.

\subsection{Equidistribution towards the bifurcation current}\label{sec_intro_equi}
As mentioned above, 
one of our main tools in the proof of Theorem \ref{teo_new_bif} 
is 
an approximation result for the bifurcation current by means of dynamically defined 
hypersurfaces in the parameter space.
In dimension 1, the
idea of seeing $\tbif$ 
 as a limit of currents detecting dynamically interesting parameters
 goes back to
Levin \cite{levin1982bifurcation} (see also \cite{levin1990theory}), who proved that
 the centres of the hyperbolic components 
of the Mandelbrot set equidistribute the bifurcation current, which is supported on its boundary.
This result was later generalized in order to cover any family of polynomials
(and actually rational maps)
\cite{bassanelli2011distribution,okuyama2014equidistribution}, the distribution
of maps
 with a cycle of any given multiplier \cite{bassanelli2011distribution,buff2015quadratic,gauthier2016equidistribution,gauthier2017hyperbolic}
or with
preperiodic critical points \cite{dujardin2008distribution,favre2015distribution}.

In our situation, in the proof of Theorem \ref{teo_new_bif}
we need an equidistribution property towards $\tbif$
of
the parameters admitting a periodic point with \emph{vertical} multiplier $\eta$.
 Since the same techniques allow
to prove a general result valid for any family
of endomorphisms of $\P^k$, in any dimensions $k$,
we give a full proof of this
in the Appendix
\ref{section_equidistribution}. 
The following is 
also one of our main results: it
is the first equidistribution result in the parameter space for holomorphic dynamical systems in dimension larger than one.

\begin{teointroletter}\label{teo_new_equidistribution}
	Let $(f_\lam)_{\lam \in M}$ be the family 
	of all holomorphic endomorphisms of $\P^k$ of a given degree $d\geq 2$.
	For all $\eta \in \C$ outside of a polar subset, we have
	\[\frac{1}{d^{2n}} [\per_n(\eta)] \to \tbif,\]
	where
	$\per_n (\eta):=\{ \lambda \colon 
\exists z \in J_{f_\lam} \mbox{of exact period n for }  f_\lambda \mbox{ and such that } 	\jac_z f_\lam = \eta
	\}$.
\end{teointroletter}

The general strategy of the proof of Theorem \ref{teo_new_equidistribution}
follows the main line of
the one dimensional case and
is based of techniques and tools from pluripotential theory.
However, one of the  difficulties we have to face here
is the possible presence
of infinitely many non-repelling cycles for an endomorphism of $\P^k$ -- something
which is excluded for $k=1$ by a Theorem
due to Fatou. We thus need more
quantitative estimates on the number
of repelling cycles with small multiplier, 
which are related to 
the
 approximation formula for the
Lyapunov exponent valid in any dimension
established in \cite{berteloot2008normalization}.

\subsection{Organization of the paper}

After recalling the notions of vertical expansion, stability and bifurcation and fixing the
notations in Section \ref{section_prelim}, in Section \ref{section_skew_things}
we prove our approximation formulas for the vertical Lyapunov exponent. This motivates the study of \emph{vertical bifurcations}. 
Theorems \ref{teo_new_bif} and \ref{teo_new_hyp} are proved in Sections \ref{section_new_bif} and \ref{section_vertical_hyp}. 
Theorem \ref{teo_new_equidistribution} is proved in the appendix, together with its adapted
version
for families of polynomial skew product needed in the proof of Theorem \ref{teo_new_bif}.

\section{Preliminaries and notations}\label{section_prelim}

\subsection{Polynomial skew products}

A
polynomial skew product is an endomorphism of $\P^2$
 of the form
$f(z,w)= (p(z), q(z,w))$, for $p,q$ polynomials.
 The second coordinate will be also written as $q_z (w)$.
We shall denote by $z_j:= p^j (z)$ the points of 
the orbit
of  $z\in \C$ under the base polynomial $p$.
In this way, we can write
\begin{equation}\label{eq_for_qn}
f^n (z,w) = (p^n (z) , q_{z_{n-1}} \dots q_{z_1} \circ q_z (w)) =: (z_n,Q^n_z (w) ).
\end{equation}
The dynamics of 
polynomial skew products has been studied in detail by Jonsson \cite{jonsson1999dynamics}.
In particular, he proved that it is possible to associate to each $z \in J_p$ a \emph{vertical Julia set}
$J_z$, defined as the boundary of the set of points
 that have bounded orbit under the 
sequence $Q^n_{z}$. The map $z\mapsto J_z$ is lower semicontinuous. The following result describes the structure
of the \emph{Julia set} $J_f$ of $f$, i.e., the support of its measure of maximal entropy $\mu_f$.

\begin{teo}[Jonsson \cite{jonsson1999dynamics}]\label{teo_jonsson_julia}
Let $f$ be a polynomial skew product. Then
$J_f = \bar{\bigcup_{z\in J_p}\{z\}\times  J_z}.$
Moreover, $J_f$ is the closure of the repelling periodic points for $f$.
\end{teo}

\subsection{Vertical expansion}\label{s:vertical}
Recall that an
endomorphism $f$ of $\P^k$ 
is \emph{hyperbolic} or \emph{uniformly expanding on the Julia set}
if
there exist
constants $c>0, K>1$ 
such that, for every $x\in J$ and $v\in T_x \P^k$, we have $\norm{Df_x^n (v)}_{\P^k}\geq c K^n$
(with respect for instance to the standard norm on $\P^k$).
In the case of polynomial skew products, this condition in particular
forces
the base polynomial $p$ to be hyperbolic.
Jonsson  thus 
introduced an adapted notion of hyperbolicity valid for any base polynomial $p$.
Given
 an invariant set $Z$ for $p$
set
\begin{enumerate}
\item $C_Z :=\cup_{z\in Z} \{z\}\times C_z$ for the \emph{critical set over $Z$},
\item $D_Z := \bar{ \cup_{\geq 1}f^n C_Z} =: \cup_{z\in Z} \{z\}\times D_{Z,z}$ for the \emph{postcritical set over $Z$}, and
\end{enumerate}
When dropping the index $Z$, we mean that we are considering $Z=J_p$.
\begin{defi}[Jonsson, \cite{jonsson1999dynamics}]
Let $f(z,w)=(p(z),q(z,w))$ be a  polynomial skew product 
and $Z\subset \C$ be such that $p(Z)\subset Z$. We say that  $f$
is \emph{vertically expanding over $Z$} if
there exist constants
$c>0$ and $K>1$
such that
$\left|\pa{Q^n_z}' (w)\right|\geq cK^n$ for every $z\in Z$, $w\in J_z$ and $n\geq 1$.
\end{defi}

For polynomials on $\C$, hyperbolicity is equivalent
to the fact that the closure of the postcritical set
is disjoint from the Julia set. In our situation, we have the following
analogous characterization.

\begin{teo}[Jonsson \cite{jonsson1999dynamics}]\label{teo_jonsson_unique}
Let $f(z,w)=(p(z),q(z,w))$
be a polynomial skew product. Then $f$ is vertically expanding over $Z$ if and only if
$D_Z \cap J_{Z} = \emptyset$, and 
 the following conditions are equivalent:
\begin{enumerate}
\item $f$ is hyperbolic;
\item $D\cap J = \emptyset$;
\item $p$ is hyperbolic, and $f$ is vertically expanding over $J_p$.
\end{enumerate}
\end{teo}

\subsection{Stability and bifurcations}

The definition and study of the notions
of stability and bifurcations for endomorphisms of projective spaces of
any dimension
is given in \cite{bbd2015,b_misiurewicz}, see also
 \cite{bb_warsaw}.
Since
  we will be mainly concerned with families of polynomial skew products in dimension 2,
we cite 
an adapted version 
in our setting.

\begin{teo}[\cite{bbd2015}]\label{teo_bbd_skew}
Let $(f_\lam)_{\lam\in M}$ be a holomorphic family of polynomial skew products of degree $d\geq 2$.  Then the following are equivalent:
\begin{enumerate}
\item the repelling cycles move holomorphically;
\item $dd^c_\lam L (\lam)\equiv 0$;
\item there are no \emph{Misiurewicz parameters}.
\end{enumerate}
\end{teo}

We say that a family is \emph{stable} if any of the conditions
 above is satisfied. 
The holomorphic motion of the repelling cycles is defined as in dimension 1 (see e.g. \cite{berteloot2013bifurcation,dujardin2011bifurcation},
or \cite[Definition 1.2]{bbd2015}
in this context). Notice that, is our setting, all repelling points are contained in the Julia set of $f_\lam$
 see Theorem \ref{teo_jonsson_julia}.
 We denote by $L(\lam)$
 the sum of the Lyapunov exponents of $f_\lam$, with respect to $\mu_\lam$, which 
 is a 
plurisubharmonic (psh for short)
 function on the parameter space. Thus, $dd^c_\lam L$
is a positive closed (1,1) current on $M$.
 We call it the \emph{bifurcation current} for the family. Its support
is the \emph{bifurcation locus}.
A \emph{Misiurewicz parameter} is
 a generalization
to any dimension
of a map with a critical point that is
(non-persistently) preperiodic to a repelling cycle. More precisely, they are defined as follows:

\begin{defi}\label{defi_misiurewicz}
Let $(f_\lam)_{\lam \in M}$
be a holomorphic family of endomorphisms of $\P^k$
and let $C_F$ be the critical
set of the map $F(\lam,z):= (\lam, f_\lam(z))$.
A point $\lambda_0$ of the parameter space $M$ is called a \emph{Misiurewicz parameter} if
there exist
a neighbourhood
$N_{\lambda_0} \subset M$ of $\lambda_0$ and
a holomorphic map $\sigma \colon N_{\lam_0}\to \C^k$
 such that:
 \begin{enumerate}
 \item\label{defi_in_julia_rep} for every $\lambda\in N_{\lam_0}$, $\sigma(\lambda)$ is a
 repelling periodic point;
 \item\label{defi_in_julia} $\sigma(\lam_0)$ is in the Julia set $J_{\lam_0}$ of $f_{\lam_0}$;
 \item\label{defi_capta_critico} there exists an $n_0$ such that $(\lambda_0, \sigma (\lambda_0))$ belongs
 to some component of $F^{n_0} (C_F)$;
 \item\label{defi_intersez} $\sigma(N_{\lambda_0})$ is not contained in a component of $f^{n_0} (C_F)$
 satisfying \ref{defi_capta_critico}.
 \end{enumerate}
\end{defi}

\subsection{Non autonomous bifurcations}\label{section_non_autonomous}

Given a
family of polynomial skew products of the form
$ f_\lam (z,w) = (p(z), q_{\lam,z} (w))$, $\lam \in M$,
for every $z\in J_p$ we can consider the 
non-autonomous iteration of $q_{\lam,p^j (z)}$
associated to
the fibre $z$.
The corresponding 
\emph{vertical Green function} 
is given by $G_{\lam,z} (w) = \lim_{n\to \infty} \frac{1}{n} \log^+ \big\|Q^n_{\lam,z} (w)\big\|$
and is psh.
The proof of the following results are completely analogous to the autonomous case.

\begin{prop}
Let $c(\lam)$ be a (marked) critical point of $q_{\lam,z}$.
The family $Q^n_{\lam,z} (c (\lam))$ is normal if and only if
 $dd^c_\lam G_{\lam} (z,c(\lam))\equiv 0$. 
\end{prop}

\begin{defi}\label{defi_bz}
We denote by $\bcal_{z,c} :=\set{\lam \colon G(\lam,z,c(\lam))=0}, \tbifzc:= dd^c_\lam G(\lam,z,c(\lam))$,
 and $\Bif_{z,c}:=  \Supp \tbifzc$ the \emph{boundedness locus}, the \emph{bifurcation current} and the \emph{bifurcation
locus} 
associated to a marked critical point $c$ in the fibre $z$. 
$\bcal_{z}, \Bif_{z}$ and $\tbifz$
are
the unions (or the sum) of the sets (currents) above, for $c$ critical point for $q_{z}$.
\end{defi}

\begin{lemma}\label{lemma_semicont}
For all $z$ and $c$ we  have $\Bif_{z,c}= \partial \bcal_{z,c}$ and $\Bif_{z}= \partial \bcal_{z}$.
For every compact $M'\Subset M$
the set $M' \cap \mathcal{B}_z$  (resp., $M' \cap \Bif_z$)
varies upper (resp. lower) 
semicontinuously with $z$.
\end{lemma}

\subsection{Quadratic skew products}\label{section_quadratic}

We now specialize to
 quadratic polynomial skew products. The general form is 
\[
(p(z), Az^2 + Bzw + Cw^2 + Dz + Ew + F),\]
where $p$ is a quadratic polynomial. Notice that we necessarily
have $C\neq 0$ in order to extend the map above to an endomorphism to $\P^2(\C)$.

\begin{lemma}\label{lemma_simplify_quadratic}
Every quadratic skew product with $p(z)$
as first component
is affinely conjugated
to
a map of the form
$(z,w) \mapsto (z^2+d, w^2 + az^2 + bz +c)$.
\end{lemma}

We can thus
 consider the space $\skp$ of quadratic skew products over the 
base $p$ as identified with $\C^3$.
We will also work with the compactification of 
$\skp$ as $\mathbb P^3$, and denote by $\ptwo_\infty$ the hyperplane at infinity.

Notice that, for all maps in $\skp$,
the fiber at any $z$ contains a unique critical point for $q_z$, $w=0$.
By the results of the previous section, in order to understand
the stability of the family we then need to study
the Green function at the points of the form $(z,0)$, with $z\in J_p$. This leads
to the following definition.

\begin{defi}
	We partition the parameter space $\skp$
 as follows:
	\begin{enumerate}
		\item $\ccal := \{ \lambda \in \C^3 : \forall z \in J_p, G(z,0)=0 \} = \bigcap_{z\in J_p} \mathcal{B}_z$;
		\item $\dcal := \{ \lambda \in \C^3 : \forall z \in J_p, G(z,0)>0 \} = \bigcap_{z\in J_p} \mathcal{B}_z^c$;
		\item $\mcal := \C^3 \backslash (\ccal \cup \dcal)$.
	\end{enumerate}
\end{defi}

In the case where $J_p$ is connected, $\ccal$ is (in restriction to our family)
what in \cite{jonsson1999dynamics} is called the \emph{connectedness locus}, meaning the set of parameters such that
$J_p$ is connected, and  $J_z$ is connected
for all $z \in J_p$.
$\ccal$ is closed and $\dcal$ is open. 
It follows from \cite{jonsson1999dynamics} that $\dcal$ is in fact a union of vertically expanding components (hyperbolic, if $p$ is hyperbolic).
As we will see below,
$\ccal$ is bounded  (Corollary \ref{coro:3bzbounded})
but $\mcal$ and $\dcal$ are not (and actually contain
unbounded hyperbolic components, see Sections \ref{section_hyp_D} and \ref{section_hyp_other}).

We further define $\Dd'$ as the subset of $\Dd$ with access to infinity:
\begin{equation}
\Dd':=\{\lam \in \Dd : \text{ there exists a path joining }\lam \text{ to } \ptwo_\infty \backslash E \text{ in }\Dd \}.
\end{equation}
Note that connected components of $\Dd'$ are also connected components of $\Dd$.

\begin{remark}
 Dujardin \cite{dujardin2016non} and Taflin \cite{taflin_blender}  recently proved
that some polynomial skew products are in the interior of the bifurcation locus, 
a phenomenon that contrasts with the one-variable situation. 
Such behaviour can only occur in $\mcal$: indeed, parameters in $\dcal$ are vertically expanding hence in the 
stability locus. As for parameters in $\ccal$, 
any connected component of $\mathring{ \ccal}$
is a stable component. 
\end{remark}

For technical reasons, 
we will also need to consider  some
1-codimension subfamilies of
$\skp$ defined as follows.
Given any $\alpha = (\alpha_1, \alpha_2, \alpha,_3)\in \C^3$, we denote by $\skpa$ 
the family corresponding to the 
  hyperplane
$\alpha_1 a + \alpha_2 b + \alpha_3 c =0$
in the parameter space $(a,b,c)$.
We denote by $\Bif^\alpha$ and $\tbif^\alpha$ the bifurcation locus and current in the family
$\skpa$.
The definition of $\bcal_z, \Bif_z, \tbifz$ can also
easily be adapted for $\skpa$, and
we can define
 $\bcal_z^\alpha, \Bif_z^\alpha,{\tbifz^\alpha}$ in an analogous way as
 in Definition \ref{defi_bz}.
We denote by $E^\alpha$ the hyperplane $\set{[a,b,c]\colon\alpha_1 a + \alpha_2 b + \alpha_3 c =0}\subset \P^2_\infty$.

\section{Lyapunov exponents and fiber-wise bifurcations}\label{section_skew_things}

In this Section we establish decomposition formulas for the bifurcation current and locus.
 Both will 
be used in the proof of Theorem \ref{teo_new_bif}, and they
motivate the classification in Theorem \ref{teo_new_hyp}.

\subsection{The vertical bifurcation}
Consider 
 a family of polynomial 
 skew products of $\C^2$
  of the form $f_\lam (z,w)= (p_\lam(z), q_{\lam,z} (w))$.
By
\cite{jonsson1999dynamics},
the two Lyapunov exponents of $f_\lam$ with respect to its maximal entropy measure 
are equal to
\begin{equation}\label{eq_lyapunov}
L_p(\lam) 
 =\log d + \sum_{z\in C_{p_\lam}} G_{p_\lam} (z) \quad \mbox{ and }\quad
L_v(\lam)
  = \log d + \int 
  \Big(
  \sum_{w\in C_{\lam,z}} G_\lam (z,w)
\Big)  
   \mu_{p_\lam},
\end{equation}
where $C_{p_\lam}$ and $\mu_{p_\lam}$ are the critical set and the equilibrium measure of $p_\lam$
and $C_{\lam, z}$ is the critical set of $q_{\lam,z}$.

By
\cite{pham,ds_cime}
 the sum $L(\lam)= L_p (\lam)+L_v (\lam)$
is a 
 psh function.
In our situation,  
we are interested in the two functions 
$L_p$
 and $L_v$ separately.
 The first
  is
   psh, since it is the Lyapunov function of
a polynomial family on $\C$. The following result ensures that $L_v$ enjoys the same property.

\begin{prop}\label{prop_vertical_bif}
Let $(f_\lam)_{\lam \in M}$
be a  holomorphic family of polynomial skew product. The map $\lam\mapsto L_v (\lam)$ is psh. In particular, the current
$\tbif^v := dd^c L_v = \tbif - \tbif ({p_\lam})$ is positive and closed.
\end{prop}

The Proposition above is a direct consequence of the (pointwise and $L^1_{loc}$)
convergence of the first sequence 
in the following Lemma. 
We denote by $\Rr_N (\lam)\subset \Pp_N (\lam)$ the  sets
\[
\Pp_N (\lam)  := \set{z\in \C\colon p_\lam^n(z)=z}
\quad \mbox{ and } \quad
\Rr_N (\lam) := \set{z\in \C\colon p_\lam^n(z)=z, \abs{(p_\lam^n)'(z)}>1}.
\]

\begin{lemma}\label{lemma_formula_Lv}
 We have
 \[
 L_v (\lam) = \lim_{N \to \infty} \frac{1}{d^{N} }\sum_{z\in \Pp_N ( \lam)} \sum_{w\in C_{\lam,z}} G_\lam (z,w)
 = \lim_{N \to \infty} \frac{1}{ d^{N} }\sum_{z\in \Rr_N (\lam)} \sum_{w\in C_{\lam,z}} G_\lam (z,w)
 \]
 where the convergence is pointwise and in $L^1_{loc} (M)$. 
\end{lemma}

\begin{proof}
The pointwise convergence
of both sequences
  follows from the equidistribution of the periodic  (or repelling periodic)
 points towards the equilibrium measure $\mu_{p_\lam}$ for the polynomial $p_\lam$, 
 and the continuity of the Green function.
 
 The continuity of $G$ and the fact that $\card \Pp_N (\lam), \card \Rr_N (\lam) \leq d^N$
  for every $N\in \N$ also imply that
  both sequences are locally uniformly bounded.
It thus suffices to prove that the first sequence consists of psh functions to get that
there exist
$L^1_{loc}$ limits.
 By the previous part the only possible limit will then be
$L_v (\lam)$, proving the statement.
In order to do so, since $G$ is psh, it suffices to notice that the set $C_N$ given by
$C_N := \set{ (\lam, z, w) \colon z\in \Pp_N (\lam), w\in C_{\lam,z}}$
is an analytic subset of $M\times \C^2$. The assertion follows.
\end{proof}

In view of Proposition \ref{prop_vertical_bif}, 
and since $\Bif (p) \equiv \Supp dd^c L_p$, we can 
decompose
the bifurcation locus
$\Bif (F)$ as a union (non necessarily disjoint)
\[
\Bif (F) = \Bif (p) \cup \Bif (q)
\]
where we denoted $\Bif (q):= \Supp \tbif^v = \Supp dd^c L_v$.
We will call such bifurcations \emph{vertical}.
Our next goal
 consists in getting a better understanding of the set $\Bif (q)\setminus \Bif (p)$.

\subsection{A decomposition for the bifurcation current and locus}

\begin{teo}\label{teo_new_skew_things}
Let $f_\lam(z,w) = (p(z),q_\lam(z,w))$, $\lam \in M$, be a holomorphic 
family  of polynomial skew products of degree $d$.
Then
\[
\tbif = \int_{z\in J_p} \tbifz \mu_{p} 
\quad \mbox{ and } \quad 
\Bif (F)= \bar{ \bigcup_{z\in J_p} \Bif_z }.
\]
\end{teo}

\begin{proof}
The first formula follows from
the expression for $L_v$
in \eqref{eq_lyapunov}
The inclusion
$\subseteq$
 in the second formula is an immediate consequence of the first formula.

By 
 the lower semi continuity of $z\mapsto \Bif_{z}$, in order to  prove the
 reversed inclusion
 it is enough to show that,
for every $N$
and $z\in \Rr_N$, we have
$\Bif (F) \supseteq  \Bif (Q^N_{\lam,z}),$ 
 where $Q^N_{\lam,z}$ denotes the 1-dimensional family $(\lam,w) \mapsto (\lam, Q^N_{\lam,z} (w))$.
In order to prove this, given any such $N$ 
 and $z$,
   let us consider a parameter $\lam_0 \in \Bif(Q^N_{\lam,z})$.
 There exists a parameter $\lam_1$ close to $\lam_0$
 which is Misiurewicz for the family $Q^N_{\lam,z}$,
  i.e., there exist a critical point $c$ for $Q^N_{\lam_1,z}$,
 a number $N_0\geq1$ and a $N_1$-periodic repelling point $w$
 for $Q^N_{\lam_1,z}$ 
 such that $Q^{NN_0}_{\lam,z} (c)= w$, 
 and the relation $Q^{NN_0}_{\lam,z} (c(\lam))= w(\lam)$
 does not persistently
 hold for every $\lam$ near $\lam_1$. 
Here  $c(\lam)$ and $w(\lam)$ are the local holomorphic motions of $c$ and $w$ as a critical point and
 as a $N_1$-periodic repelling point, respectively
 (we assume for simplicity that we can mark the critical point $c$, the argument being similar otherwise).
 The point $(z,c)$ 
 is in particular critical also for $f_{\lam_1}$, 
 and the point $(z,w)$ is $NN_1$-periodic and repelling for $f_{\lam_1}$.
 So, it is enough to check that
 there does not exist any holomorphic map 
 $\lam \to (z(\lam),\t c(\lam)) \in C(F_\lam)$ such that $(z(\lam_1),\t c(\lam_1)) = (z, c)$
 and
 the relation $F^{NN_0}_\lam \pa{ z(\lam),\t c(\lam)}= \pa{z,w(\lam)}$ holds persistently
 in a neighbourhood of $\lam_1$.
 First of all, by the finiteness of the $p^{N_0}$-preimages of $z$, 
 up to restricting ourselves to a small neighbourhood of this point,
 we can assume that every
 $F^{N_0}$-preimage of $(z,w(\lam))$ belongs to the fiber of $z$, too. 
 In this way, any persistent critical relation 
 must happen in the fibers of $z$.
 This in excluded, since the parameter is Misiurewicz for the restricted family.
\end{proof}

\subsection{Approximations for the bifurcation current: periodic fibres and preimages}

We  characterize here the Lyapunov exponents
of a skew product map by means of the Green functions of the return maps
of the periodic vertical fibers.
This allows us to approximate
the bifurcation current by means of
the bifurcation currents
of these return maps.
We  fix  the base polynomial
$p$ for simplicity, but the results are generalizable 
to families
where also $p$ is allowed to depend from a parameter, see also \cite{dupont2018dynamics}.

\begin{prop}
Let $f_\lam (z,w)=(p(z),q_\lam(z,w)), \lam \in M$ be a family of polynomial skew products. 
Then
 \begin{equation}\label{eq_Lv_lim_green}
 L_v (\lam)
 = \lim_{N\to \infty} \frac{1}{N d^{N} }\sum_{z\in \Rr_N} \sum_{w\in C \pa{Q^N_{\lam,z}}} G_{Q^N_{\lam,z}} (w)
 \end{equation}
 where the convergence is pointwise and in $L^1_{loc} (M)$. In particular,
 \[
\tbif^v =  \lim_{N\to \infty} \frac{1}{N d^{N} }\sum_{z\in \Rr} \tbif (Q^N_{\lam,z}).
 \]
\end{prop}

\begin{proof}
By Lemma \ref{lemma_formula_Lv} 
we only have to prove that, for any $\lam,N$ and $z \in \Rr_N$,
\begin{equation}\label{eq_1N}
 \sum_{w\in C_{\lam,z}} G_\lam (z,w)=\frac{1}{N}
 \sum_{w\in C \pa{Q^N_{\lam,z}}} G_{Q^N_{\lam,z}} (w).
\end{equation}
 First
  notice that, for every $\lam, N,  z \in \Rr_N$
 and $w\in \C$,
 we have
  $G_\lam (z, w) = G_{Q^N_{\lam,z}} (w)$.
 So, the
 left hand side of \eqref{eq_1N} is equal to
 $
  \sum_{w\in C_{\lam,z}} G_{Q^N_{\lam,z}} (w).
 $
 We are thus left with checking that, for a given skew product $f(z,w)= (p(z), q_z(w))$, for every $N$-periodic point $z$ of $p$, we have
 \[
 \sum_{j=0}^{N-1} \sum_{w\in C_{p^j (z)}} G_{Q^N_{p^j (z)}} (w) =
 \frac{1}{N} \sum_{j=0}^{N-1} \sum_{w\in C \pa{Q^N_{p^j(z)}}} G_{Q^N_{p^j (z)}} (w). 
 \]
 Let us first describe the critical set of $Q^N_{p^j (z)}$, that we denote by $C^j$. Since $Q^N_{p^j (z)}$ is by definition
 equal to $q_{p^{j-1} (z)} \circ \dots \circ q_{z} \circ q_{p^{N-1}(z)}\circ \dots \circ q_{p^{j+1}(z)}\circ q_{p^j (z)}$, we have
 $C^j = C\pa{Q^N_{p^j (z)}} = \bigcup_{i=0}^{N-1} C^j_{i}$, where
 \[
 \begin{aligned}
 C_0^j & = C(q_{p^j (z)})\\ 
 C_1^j & = q_{p^j (z)}^{-1} C(q_{p^{j+1}(z)})\\
 C_2^j & = q_{p^j (z)}^{-1} q_{p^{j+1}(z)}^{-1} C(q_{p^{j+2}(z)}) = \bra{Q^2_{p^j (z)}}^{-1} C(q_{p^{j+2}(z)})\\
 \vdots &\\
 C_{N-1}^j &= q_{p^j (z)}^{-1} q_{p^{j+1}(z)}^{-1} \dots q_{p^{j-2}(z)} C(q_{p^{j-1} (z)}) = \bra{Q^{N-1}_{p^j (z)}}^{-1}C\pa{q_{p^{j-1}(z)}}
 \end{aligned}
 \]
 (each term $C^k_l$ is to be thought of as a subset of the fibre 
 over $p^k (z)$ which is the preimage of the critical set of
 $q_{p^{k+l (\mod N)} (z)}$ by $Q^l_{p^k (z)}$).
 So, it suffices to prove that, for any $0\leq j,i\leq N-1$, we have
 \[\sum_{w\in C^j_0 = C \pa{q_{p^j (z)}}} G_{Q^N_{p^j (z)}} (w) =
 \sum_{w \in C^i_{j-i} = \bra{Q^{j-i}_{p^{i}(z)}}^{-1} \pa{C \pa{q_{p^j (z)}}}} G_{Q^N_{p^i(z)}} (w),\]
where $j-i$ has to be taken modulo $N$. 
 But this follows from the identity
 $G(f (\cdot))= d G(\cdot)$.
Indeed, for points $(p^l (z),w)$ in the fibre $\set{p^l(z)}\times \C$, we have $G_f(p^l (z),w)= G_{Q^N_{p^l (z)}} (w)$.
Moreover $\set{p^i (z)} \times C^i_{j-1}$ contains exactly $d^{j-i}$ preimages (counting multiplicities)
 by $F^{j-i}$ of any point in
$\set{p^j(z)}\times C_0^j$ (out of the total $d^{2(j-i)}$, since we do not consider preimages other than the ones contained in the fiber over $p^j (z)$).
So, for each $w\in C^j_0$ in the left sum, with value $G_{Q^N_{p^j (z)}} (w)$ , there are $d^{j-i}$ preimages $w_1 , \dots, w_{d^{j-i}}$
in the right sum, each one with value $G_{Q^N_{p^i(z)}} (w_l) = G_{Q^N_{p^j (z)}} (w) / d^{j-i}$.
 The assertion follows.
\end{proof}

By similar arguments, by exploiting
the equidistribution of preimages
of generic points, we can establish the following further approximation of the bifurcation current and locus.

\begin{prop}
Let $f_\lam(z,w)=(p(z), q_\lam(z,w)), \lam \in M,$ be a holomorphic family of
polynomial
skew products of degree $d\geq 2$.
Let $z \in J_p$.
\[
\tbif = \lim_{N\to \infty} \frac{1}{d^N} \sum_{y\colon p^N (y)=z} T_{\mathrm{bif},y} \mbox{ and }
\Bif (M) = \bar{ \bigcup_{N \in \N} \bigcup_{y \colon p^N (y)=z} \Bif{y}}.
\]
\end{prop}

\section{Bifurcations near $\P^2_\infty$ and Theorem \ref{teo_new_bif}}\label{section_new_bif}

\subsection{Accumulation at $\P^2_\infty$}
In this Section we prove the first part of Theorem \ref{teo_new_bif}. Since we will need it in the next Section
\ref{section_current_infinity},
we  actually
characterize the accumulation of the bifurcation locus also for generic subfamilies of the form $\skpa$, see Section \ref{section_quadratic}.
Namely, we will consider in the following families 
$\skpa$ corresponding to $\alpha \in \C^3$ satisfying
\begin{equation}\label{eq_condition_subfamily}
[\alpha_1, \alpha_2, \alpha_3]\neq [z^2,z,1] \mbox{ for all } z \in J_p.
\end{equation}

Condition \eqref{eq_condition_subfamily}
 means that the line at infinity $E^\alpha$
of the family $\skpa$
 is different 
from any line $E_z$ 
corresponding to any $z\in J_p$. Notice in particular that, for every $[a,b,c] \in \P^2_\infty$,
among all the families such that $[a,b,c]\in E^\alpha$, at most two 
do not satisfy condition \eqref{eq_condition_subfamily}.

\begin{teo}\label{teo_acc_infinity_bif}
In the family $\skp$ (resp., $\skpa$
for any $\alpha$ satisfying in \eqref{eq_condition_subfamily}),
 the following hold.
\begin{enumerate}
		\item
	For every $z \in J_p$,
		the cluster set at infinity of 
		$\mathcal{B}_z$ and $\Bif_z$ (resp., $\bcal_z^\alpha$ and $\Bif_z^\alpha$) is exactly
		$E_z$ (resp., $E_z \cap E^\alpha$).
		\item
		The cluster set at infinity of $\mathrm{Bif}$ (resp., $\Bif^\alpha$)
		is exactly
		$E$ (resp., $E\cap E^\alpha$).
	\end{enumerate}
\end{teo}

It will be useful to fix the following notations.
For $\lambda:=(a,b,c) \in \C^3$ and $p$ a monic quadratic polynomial, we set
$$f_\lambda(z,w)=(p(z),w^2+az^2+bz+c).$$
For $\lambda \in \C^3$, recall that  
$q_{\lambda,z}(0)=az^2+bz+c$ and set
$R(f_\lambda):=\sup_{z \in J_p} |az^2+bz+c|$.
We then have the following elementary lemma.

\begin{lemma}\label{lem:adherenceinclus}
	For all $\lam$
	with $\|\lam\|$ sufficiently large, for all $z_0 \in J_p$,  if  $G_\lambda(z_0,0)=0$
	then
	$|q_{\lambda,z_0}(0)| \leq 2 \sqrt{R(f_\lambda)}$. 
	In particular, the cluster set of $\mathcal{B}_{z_0}$
	in $\P^2_{\infty}$
	is contained in
	$E_{z_0}$.
\end{lemma}

\begin{proof}
	For $n \in \N$, set $z_n:=p^n(z_0)$ and $\rho_n:=az_n^2+bz_n+c$.
	Let $w_n:=Q_{z_0}^{n+1}(0)$: then $w_0=\rho_0$
	 and $w_{n+1}=w_n^2+\rho_n$.
	Therefore we have $|w_{n+1}| \geq |w_n|^2 - R(f_\lambda)$,  
	and since by assumption $(w_n)_{n \in \N}$ is bounded, 
	if $R(f_\lambda)\geq 1$ (which is true for $\|\lam\|$ sufficiently large)
	we must have  
	$|w_n| \leq 2 \sqrt{R(f_\lambda)}$
	for all $n \in \N$. The first assertion follows by taking $n=0$.
	
	Take now $(a_0,b_0,c_0)$ such that
	$[a_0,b_0,c_0]\notin E_{z_0}$. In particular, we have  $|a_0 z^2_0+ b_0z_0 + c_0 |>2 \eps_0$
	for some positive $\eps_0$, which implies that $|a z^2_0+ bz_0 + c |> \eps_0$ for all
	$(a,b,c)$ sufficiently close to $(a_0,b_0,c_0)$.
	It follows that for all such $(a,b,c)$ 
	both $q_{z_0,t (a,b,c)} (0)$
	and $R(f_{t(a,b,c)})$
	 grow linearly in $|t|$ as $|t|\to\infty$.
	By the first part of the statement, this implies that $G_{t(a,b,c)}(z_0)>0$
	for all $(a,b,c)$ sufficiently close to $(a_0,b_0,c_0)$ and $t$ with $|t|$ large enough. The assertion follows.
\end{proof}

\begin{proof}[Proof of Theorem \ref{teo_acc_infinity_bif}]
		By Lemma \ref{lem:adherenceinclus}, the cluster set of $\mathcal{B}_z$ is included 
		in 
$E_z$.		We thus prove the opposite inclusion.
		We first consider $z$ such
		that $z=p^n(z)$. Since
$E_z$		
		is an irreducible curve
		(more precisely, a projective line), it is enough to note that there is a component $C$ of 
		$\per_n(0)$ such that for all $\lambda \in C$, $f_\lambda^n(z,0)=(z,0)$.
		Indeed, that component $C$ intersects the plane at infinity in some (1 dimensional)
		hypersurface that is contained in 
		$E_z$ and is therefore
		equal to $E_z$. Moreover, it is clear that 
		$C \subset \mathcal{B}_z$.
		
		{
		Let us now pick any (non necessarily periodic)  $z\in J_p$. It
		is enough to prove the statement for the family $\skpa$
		for all $\alpha$ satisfying \eqref{eq_condition_subfamily}. Since 
		$E_z\cap E^\alpha$ is a single point, because of the inclusion already proved
		it is enough to show that $\mathcal{B}_z^\alpha$ is not compact. This follows by the continuity of the Green function
		and the fact that for the dense subset of periodic points the corresponding $\mathcal{B}_z^\alpha$ 
		is not compact, as proved above. The assertion for
		$\Bif_z$ also easily follows.}

		Let us now prove the second assertion.
		Observe that $\mathrm{Bif} \subset \bigcup_{z \in J_p} \mathcal{B}_z$.
		Therefore, the cluster set at infinity of $\mathrm{Bif}$ is contained in 
		the cluster set of $\bigcup_{z \in J_p} \mathcal{B}_z$, which a priori might be larger than the 
		union of cluster sets of $\mathcal{B}_z$; but the estimate from Lemma \ref{lem:adherenceinclus} 
		implies that this is not the case.
		Indeed, let $(a_n,b_n,c_n)_{n \in \N}$ be a sequence of points in $\bigcup_{z \in J_p} \mathcal{B}_z$
		going to infinity, and such that 
		$[a_n,b_n,c_n]\to [a,b,c] \in \ptwo_\infty$. 
		For each $n$ there is at least one
		$z_n \in J_p$ 
		such that $(a_n,b_n,c_n) \in \bcal_{z_n}$, and thus, by Lemma \ref{lem:adherenceinclus},
		$$|a_n z_n^2+b_n z_n+c_n| \leq 2 \sqrt{\sup_{z \in J_p} |a_n z^2+ b_n z+c_n| }.$$
		Since $(a,b,c) \mapsto \sup_{z \in J_p} |az^2+bz+c|$ is a vector space norm		
		on $\C^3$, there is 
		some constant $C_p>0$ such that for all $(a,b,c) \in \C^3$,
		$$\frac{1}{C_p} \|(a,b,c)\|_\infty \leq  \sup_{z \in J_p} |az^2+bz+c| \leq C_p  \|(a,b,c)\|_\infty$$
		and therefore, setting $M_n:=\|(a_n,b_n,c_n)\|_\infty$, we have
		$$\left|\frac{a_n}{M_n} z_n^2 + \frac{b_n}{M_n} z_n + \frac{c_n}{M_n}\right| \leq 2 C_p\sqrt{\frac{1}{M_n}}.$$
	Passing to the limit, we conclude that $z \mapsto az^2+bz+c$ must vanish at least once on $J_p$.
		
		This takes care of one inclusion.
		Now let us prove that the cluster of the bifurcation locus 
		on the hyperplane at infinity contains  the set $E$.
		Take $[a,b,c]\in E$, 
		so that
		 $az^2+bz+c=0$
		 for some $z \in J_p$.
		By Theorem \ref{teo_new_skew_things}, 
		we know that $\partial \bcal_z \subset \mathrm{Bif}$ (here, the boundary is taken in $\C^3$). By the first item,
		we know that $\partial \bcal_z$ accumulates on 
		$\ptwo_\infty$ to the set $\{[a,b,c] : az^2+bz+c=0\}$. This concludes the proof.
\end{proof}

\begin{cor}\label{coro:3bzbounded}
	Let $z_1, z_2,z_3 \in J_p$ be three distinct points. Then $\bcal_{z_1} \cap \bcal_{z_2} \cap \bcal_{z_3}$ is compact. In particular, $\ccal$ is compact.
\end{cor}	
	
\begin{proof}
	If $[a,b,c] \in \ptwo_\infty$ were accumulated by $\bcal_{z_1} \cap \bcal_{z_2} \cap \bcal_{z_3}$,
	then $aX^2+bX+c$ would have $z_1,z_2,z_3$ as roots, and we would have $a=b=c=0$,
	which is impossible. So $\bcal_{z_1} \cap \bcal_{z_2} \cap \bcal_{z_3}$ is closed and bounded
	in $\C^3$. 
	In particular, $\ccal = \bigcap_{z \in J_p} \bcal_z$ is compact.
\end{proof}

\subsection{The bifurcation current at infinity}\label{section_current_infinity}

We prove here 
the second part of Theorem \ref{teo_new_bif}.
Recall that we
are considering the family $\skp$ given by maps of the form
\[f_\lam = \pa{p(z), w^2 +az^2+ bz +c}\]
where $p$ is a fixed polynomial of degree 2, and $\lam = (a,b,c)\in\C^3$.

First of all, we prove that we can extend the bifurcation current
of the family $\skp$ (resp $\skpa$, see Section \ref{section_quadratic}) 
to the compactification
$\P^3$ (resp. $\P^2$)
of the parameter space
(see also \cite{berteloot2015geometry} for an analogous result for
quadratic rational maps).
\begin{lemma}
There exists a positive closed $(1,1)-$ current $\htbif$ on $\P^3$
(resp., for every $\alpha\in \C^3$  satisfying \eqref{eq_condition_subfamily} a positive closed $(1,1)$ current $\htbif^\alpha$ on $\P^2$)
of mass  1 and such that
\begin{enumerate}
\item $\htbif|_{\C^3} = \tbif$ (resp., $\htbif^\alpha|_{\C^2}= \tbif^\alpha)$;
\item for a generic $\eta\in \C$, the sequences $4^{-n} [\Per_n^J (\eta)]$
and $4^{-n} [\Per_n^v (\eta)]$
converge to $\htbif$ (resp., $\htbif^\alpha$)
 in the sense of currents of $\P^3$ (resp., $\P^2$).
\end{enumerate}
\end{lemma}

\begin{proof}
We prove the statement for the family $\skp$, the proof for $\skpa$ 
being completely analogous.

The existence of $\htbif$ follows by an application of Skoda-El Mir Theorem.
Indeed, by the equidistribution results in 
Appendix 
\ref{section_equidistribution},
the mass of $\tbif$ on $\C^2$ is $1$.
We thus can trivially extend $\tbif$ to $\P^3$, and the mass of the extension still satisfies $\|\htbif\|= 1$.

We now promote the equidistribution of $[\Per_n (\eta)]$ to $\tbif$ on $\C^3$
to an equidistribution to $\htbif$ on $\P^3$ (we denote by $\per_n (\eta)$ both $\per_n^J$ and $\per_n^v$,
the proof is the same).
 First recall (see Appendix \ref{section_equidistribution})
 that the $\Per_n (\eta)$ are actually algebraic surfaces on $\P^3$, of mass $\sim 4^n$.
 Thus, the sequence $4^{-n}[\Per_n (\eta)]$ gives a sequence of uniformly bounded (in mass) positive closed currents. We have to prove that any limit
 of this sequence coincides with $\htbif$.
 Let us denote by $T$ a cluster of the sequence.
 By Siu's decomposition Theorem, we have $T= S+ \beta [\P^2_\infty]$, where $S$ has no mass on $\P^2_{\infty}$.
 It follows from the description of the accumulation of the bifurcation locus given in Section \ref{section_quadratic}
 that $\beta=0$. Moreover, we have $S=\tbif$ on $\C^3$.
 This completes the proof.
\end{proof}

In order to study the trace of $\htbif \wedge [\P^2_\infty]$, we will first need to
obtain an analogous statement for the families $\skpa$.

\begin{teo}\label{teo_intersection_a0}
For any $\alpha \in \C^3$ satisfying \eqref{eq_condition_subfamily} we have
$\htbif \wedge [\P^1_\infty] = \int  [E_z \cap E_\alpha] \mu_p$.
\end{teo}

Notice that
the support of the measure in the right hand side
can be seen as the image 
of $J_p$ by a polynomial $\pi_\infty$
of degree at most 2 (and equal to 2 for generic $\alpha$). This polynomial can be explicitly computed
from the polynomial $az^2 + bz+c$ after substituting the relation on $a,b,c$ defining the family $\skpa$.
We denote by $J_{p,\infty}$ this support, and by $\mu_{p,\infty}$
the measure  $\int  [E_z \cap E_\alpha] \mu_p = (\pi_\infty)_* \mu_p$.

\begin{lemma}\label{lemma_lim_pern_mu_infty}
For a generic $\eta\in \D$ we have
$ 4^{-n} [\per_n^v (\eta)] \wedge [\P^{1}_\infty] \to \mu_{p, \infty}$.
\end{lemma}

\begin{proof}
By the equidistribution of the periodic points of $p$ towards $\mu_p$, it is enough to prove
that
\[[\per_n^v (\eta)]\wedge [\P^1_\infty] \sim 2^n (\pi_\infty)_*  \sum_{p^n (y)=y} \delta_y. \]
The sum in the right hand side can be taken with or without multiplicity.

First, notice that the support of $[\per_n^v (\eta)]\wedge [\P^1_\infty]$
is contained in the image by $\pi_\infty$
of
the union of the solution of $p^n (y)=y$.
Indeed, every $\per_n^v (\eta)$ is contained in the boundedness locus $\bcal_{z}$ of some fibre $z$
of period (dividing) $n$
(since
a periodic cycle of vertical multiplier $\eta\in \D$ attracts a critical point).
By 
Theorem \ref{teo_acc_infinity_bif},
$\bcal_z$
 precisely clusters at 
$\pi_\infty (z)$.

To conclude, it is enough
to
prove that for
every point 
$y$ of period $n$ for $p$
the Lelong number of $[\per_n^v (\eta)]\wedge [\P^1_\infty]$
at $(\pi_\infty)_* (y)$ is at least $\sim 2^n$.
Since the return map of the fibre corresponding to $y$
is of degree $2^n$, the above follows since the mass of $\Per_1 (\eta)$ in this one-dimensional family
is $\sim 2^n$.
 \end{proof}

 \begin{proof}[Proof of Theorem \ref{teo_intersection_a0}]
The good definition of the intersection follows the same argument as in \cite[Lemma 4.3]{berteloot2015geometry}.
We give it for completeness,
also to highlight that a different approach will be needed when considering the complete family.
We take any complex line $L$ intersecting $\P^1_{\infty}$ in a point disjoint from $J_{p,\infty}$.
The complement of this line is a copy of $\C^2$. Since the set $J_{p,\infty}$
is compact in this copy of $\C^2$, we can define the intersection here
by means of \cite[Proposition 4.1]{demailly1997complex}.
We then trivially
extend this intersection as zero on the line $L$.

\begin{remark}
When considering the full family, with the three-dimensional parameter space, we cannot find a line in $\P^2_{\infty}$
disjoint from $E$ (and thus decompose $\P^3$ as the union of $\C^3$ and
a hyperplane disjoint from $E$) and apply the argument above.
\end{remark}

We now prove that $\htbif \wedge [\P^1_\infty] = \mu_{p,\infty}$
 By Lemma \ref{lemma_lim_pern_mu_infty},
 it is enough to prove that
 \[
 4^{-n} [\per_n^v (\eta)]\wedge [\P^1_\infty] \to \htbif \wedge [\P^1_\infty].
 \]
The idea is the following: the main obstacle in getting the convergence above would be that some
components of $\per_n^v (\eta)$ become more and more tangent to $\P^1_\infty$
as $n\to \infty$ (possibly with some multiple of the plane at infinity in their cluster set).
 But this cannot happen,
because of
Lemma \ref{lem:adherenceinclus}.

To make 
the above precise 
we
 use the theory of \emph{horizontal positive closed currents}, introduced by
Dujardin \cite{dujardin2004henon}, see also
\cite{dinh2006geometry, pham, dujardin2007continuity}.
Recall that
a closed positive $(1,1)$-current in the product $\D \times \D$ is \emph{horizontal} if its support is contained
in
 $\D \times K$, for some $K$ compact in $\D$. We will use 
 the following result.

\begin{teo}[Dinh-Sibony \cite{dinh2006geometry}]\label{teo_ds_horizontal}
Let $\Rr$ be a closed positive horizontal $(1,1)$-current on $\D\times \D$, with support contained in $\D \times K$.
Then the \emph{slice} $\Rr_z$ of $\Rr$
is well defined for every $z \in \D$. The slices are measures on $\D$, supported in $K$, of constant mass.
If $\phi$ is a smooth psh function on $\D \times \D$
then
the function $z\mapsto \langle\Rr_z, \phi (z,\cdot)\rangle$ is psh.
\end{teo}

By the description of the cluster set of the $\mathcal{B}_z$'s given in Theorem \ref{teo_acc_infinity_bif},
we can find a biholomorphic image of a polydisc
$\Delta\subset \P^2$ such that the following hold (by abuse of notation, we think of the polydisc directly in $\P^2$):
\begin{enumerate}
\item $\{0\}\times \D \subset \P^1_\infty$;
\item there exists $K\Subset \D$ such that $\supp \htbif \cap \Delta \subset \D \times K$ and $\supp [\per_n^v (0)]\cap \Delta \subset \D \times K$ for every $n$.
\end{enumerate}
Indeed, suppose this is not true. We then find points in $\per_n^v (0)$ accumulating some point in $\P^1_\infty \setminus J_{p,\infty}$.
Since all the $\per_n^v (0)$ cluster on $J_{p,\infty}$, this contradicts
Lemma \ref{lem:adherenceinclus}

With this setting, we see that all the
$[\per_n^v(\eta)]$ and $\htbif$
are (uniformly) horizontal currents on $\Delta$. The convergence above can thus be rephrased as a convergence for the slices at 0:
\[4^{-n}[\per_n^v (\eta)]_0 \to \pa{\htbif}_0.\]
By standard arguments, the convergence can be tested against smooth psh test functions
on $\D$ .
By Theorem \ref{teo_ds_horizontal}
above we know that, for every $\phi$ smooth and psh in $\Delta$, the functions $u_n (z) := 4^{-n}[\per_n^v (\eta)]_z  (\phi (z,\cdot)) $
and $u(z):=  \pa{\htbif}_z (\phi(z,\cdot))$
are psh.
We claim that $u_n \to u$ in $L^1_{loc}$.
Indeed, the convergence of $4^{-n}[\per_n^v (0)]$ to $\htbif$
implies that of $\phi 4^{-n}[\per_n^v (\eta)]$ to $\phi \htbif$
in the product space $\Delta$. Since the projection on the first coordinate
of $\Delta$ is continuous, we have $u_n \to u$
as distributions. Thus, by \cite[Theorem 3.2.12]{hormander2007notions},
we have $u_n \to u$
in $L^1_{loc}$. This also implies that $u_n \to u$ almost everywhere.

Now, by Hartogs' Lemma
the $L^1_{loc}$ limit of a sequence of psh function is greater than or equal to
the pointwise limit.
In our case, the pointwise limit of the $u_n$ is given by
$u' (z) = \langle \lim_{n\to \infty} [\per_n^v (\eta)]_{z},\phi \rangle$.
Since $u' (0)= \langle \mu_{p,\infty}, \phi \rangle $
(by Lemma \ref{lemma_lim_pern_mu_infty}), we just need to prove that
$u' (0)\geq u(0)$.
Since $u$ is psh
 and $u=u'$ almost everywhere, we have a sequence of
$z_m \in \D$ converging to 0 and such that
$u(z_m)=u'(z_m) \to u(0)$. It is then enough to prove that the limit of the $u'(z_m)$
is equal to $u'(0)$, i.e., that
\[
\langle  \pa{\htbif}_{z_m}, \phi \rangle \to \langle \mu_{p,\infty}, \phi \rangle.
\]
Since $u(z_m)=u'(z_m)$, every limit $\nu$ of the slice measures on the left hand side
is a measure  on  $J_{p,\infty}$ of the form $(\pi_\infty)_* \nu'$, for some $\nu'$
probability measure on $J_p$.
 It is enough to prove that $\nu=\mu_{p,\infty}$.
Suppose this is not the case.
Lemma \ref{lemma_distinguish_measure} below 
gives a contradiction with the fact that
$\langle \mu_{p,\infty},\psi \rangle \leq  \langle \nu, \psi\rangle$
for every psh function $\psi$, as proved in the previous part. This completes the proof.
\end{proof}

\begin{lemma}\label{lemma_distinguish_measure}
Let $p$ be any polynomial on $\C$, $\mu$ its equilibrium measure and $\mu'$
a probability measure supported on the Julia set of $p$.
If $\mu\neq \mu'$ there exists a subharmonic function
$\psi$ on $\C$
such that
$\langle \mu,\psi \rangle > \langle \mu', \psi\rangle$.
\end{lemma}

\begin{proof}
	Let $p_{\mu'}$ and $p_\mu$ be the respective logarithmic potentials of $\mu'$ and 
	$\mu$, that is, 
	$p_\mu(z) = \int_{\mathbb{C}} \log |z-w| d\mu(w)$ and similarly for $\mu'$.
	Recall
	that the energy of a compactly supported Radon probability measure 
	$m$ is defined by $I(m) = \int_{\mathbb{C}} p_m(z) dm(z)$.
	Since $\mu$ is the equilibrium measure of the Julia set of $p$,
	we have
	that $I(\mu)>I(\mu')$ for every $\mu' \neq \mu$,
	see for instance \cite{ransford1995potential}.
	Therefore there must exist $z_0$
	such that
	$p_\mu(z_0)>p_{\mu'}(z_0)$.
	Setting
	$\psi(z)= \log |z-z_0|$, by definition of $p_\mu$ and $p_{\mu'}$
	we  have $\langle\mu, \psi\rangle> \langle\mu', \psi\rangle$.
	Thus,
	$\psi$ has the required property.
\end{proof}

We can now describe the intersection of the bifurcation current $\htbif$
with the hyperplane at infinity $\P^2_\infty$
in the full family.  This completes the proof of Theorem \ref{teo_new_bif}

\begin{teo}\label{teo_current_infinity_general}
The intersection
$\htbif \wedge [\P^2_{\infty}]$
is well defined and equal to
$\int_{z} [E_z] \mu_p(z)$
\end{teo}

\begin{proof}
We start proving that the intersection
 is well defined.
Since the support of $\tbif$
only clusters on $E=\cup_{z\in J_p} E_z$, we need only prove the statement in a neighbourhood of $E$.
Take a point $[a_0,b_0,c_0]\in E$.
There exist $z_0$ and $z_1$ (not necessarily distinct) such that $[a_0,b_0,c_0]\in E_{z_0},E_{z_1}$
but $[a_0,b_0,c_0]\notin E_z$ for every $z\neq z_0,z_1$.
To
prove that the intersection is well defined, we prove that 
$\htbif \wedge [\P^2_{\infty}]$
has locally bounded mass near
$[a_0,b_0,c_0]$. We fix local coordinates $x,y$, centred at $[a_0,b_0,c_0]$ 
and  such that
 the coordinate axis are transversal to both $E_{z_0}$ and $E_{z_1}$
at the origin.
Theorem \ref{teo_intersection_a0}
 implies that the intersection $\htbif \wedge [\P^2_{\infty}] \wedge [L]$ is well defined
for lines $L$ 
parallel (or almost parallel) to the $x$ and $y$ axis. Since all these intersections
are measures with uniformly bounded mass, the intersection
between $\htbif \wedge [\P^2_{\infty}]$
and the currents $\int_{x\in I} [L_x]$ and $\int_{y\in I} [L_y]$
are well defined, where $I$ is a small open neighbourhood of $0$, $L_x$ the line $\left\{x=\mbox{constant}\right\}$,
$L_y$ the line $\left\{y=\mbox{constant}\right\}$ and the integrations are against the standard Lebesgue measure.
This implies that the intersections between $\htbif \wedge [\P^2_{\infty}]$
and respectively $dx\wedge i d\b x$ and
$dy\wedge i d\b y$
are of locally bounded mass, and thus well defined.

We can now prove the formula in the statement.
For every $\eta$
the intersection at infinity of the current
$[\per_n^v (\eta)]$ is given by an average of currents of the form $[E_z]$,
with $z$ such that $p^n(z)=z$.
This implies that
$
\htbif \wedge [\P^2_{\infty}]
=
\int [E_z] \nu$
for some measure $\nu$ on $J_p$.
We can thus find $\nu$ by considering a slice of the current above
by a complex line corresponding to a family $\skpa$ with $\alpha$ satisfying \eqref{eq_condition_subfamily}.
The assertion then follows from Theorem
\ref{teo_intersection_a0}.
\end{proof}

\section{Vertical expansion, hyperbolicity, and Theorem \ref{teo_new_hyp}}\label{section_vertical_hyp}

In this section we prove Theorem \ref{teo_new_hyp}. The first point is proved in Section \ref{section_stab_preserves}, the second
in Section \ref{section_hyp_D}. In Section \ref{section_hyp_other} we give an example of a different kind of unbounded
hyperbolic component.

\subsection{Stability preserves hyperbolicity}\label{section_stab_preserves}
As mentioned in the introduction, a crucial point
of the one-dimensional theory of stability and bifurcations
is that stability preserves hyperbolicity. This result crucially relies on 
a characterization of stability not available in higher dimension, and it
is then an open problem whether the same would hold in this generality.
The following answers this question for any family of polynomial skew products.

\begin{teo}\label{teo_stab_hyp}
	Let $(f_\lambda)$ be a stable family of polynomial skew products, and let $\lambda_0$ be a parameter.
	\begin{enumerate}
		\item If $(f_\lambda)$ has constant base $p$ and $f_{\lambda_0}$ is vertically expanding over $J_p$,
		then for all $\lambda$, $f_\lambda$ is vertically expanding over $J_p$.
		\item If $f_{\lambda_0}$ is hyperbolic,
		then for all $\lambda$, $f_\lambda$ is hyperbolic.
	\end{enumerate}
\end{teo}

\begin{lemma}\label{lemma_for_stabhyp}
Let $f$ be a 
polynomial skew product with base $p$.
	Assume $(z,w) \in J_p \times \C$ is accumulated by $C_{J_p}$. Then
	there exists a sequence $(z_m, w_m) \in J_p \times \C$ of iterates of critical points,
	such that $(z_m,w_m) \to (z,w)$ and $z_m$ is a repelling periodic point for $p$. 
\end{lemma}

\begin{proof}
	By assumption, there is a sequence $(y_m,c_m) \in J_p \times \C$, such that $q_{y_m}'(c_m)=0$ and 
	$f^{n_m}(y_m,c_m) \to (z,w)$. Given any $\epsilon>0$, there exists $M \in \N$ such that 
	$\|f^{n_m}(y_m,c_m)  - (z,w) \| \leq \epsilon$ for all $m \geq M$. 	
	Since $f^{n_m}$ is continuous, 
	there exists $\delta_m>0$
	such that if $\|(z_m,c_m') - (y_m,c_m) \| \leq \delta_m$, then
	$\| f^{n_m}(z_m, c_m') - f^{n_m}(y_m,c_m)\| \leq \epsilon$. This implies that
	$\| f^{n_m}(z_m, c_m') - (z,w)\| \leq 2\epsilon$.		
	Since repelling periodic points are dense in $J_p$, we can find 
	$z'_m$ periodic and repelling arbitrarily close to $y_m$. We can then take 
    $c_m'$ such that $(z'_m,c_m') \in C_{J_p}$ is 
	$\delta_m$-close to $(y_m,c_m)$. The point
	$(z_m,w_m):=f^{n_m}(z_m,c_m')$ is then $2\epsilon$-close to $(z,w)$. Since $z'_m$ is periodic and repelling for $p$, the same holds
	for $z_m = p^{n_m}(z'_m)$.	 
	Since $\epsilon>0$ was arbitrary, the lemma is proved.  
\end{proof}

\begin{proof}[Proof of Theorem \ref{teo_stab_hyp}]
	Assume by contradiction
	that there exists $\lambda_1$
	such that $f_{\lam_1}$  is not vertically expanding. 
	We can replace our parameter space with any relatively compact connected open subset containing $\lam_0$ and $\lam_1$. 
	By Theorem \ref{teo_jonsson_unique},
	there exists
	$(z,w) \in J_{f(\lambda_1)}$
		such that $(z,w)$ is accumulated by the post-critical set of $f_{\lambda_1}$ over $J_p$.
		By Lemma \ref{lemma_for_stabhyp}, 
		there is a sequence $(z_m,w_m)$ of iterates of critical points such that $z_m$ is
		periodic for $p$ and $(z_m,w_m) \to (z,w)$.
	
	We first treat the case where it is possible to follow holomorphically all critical points over $J_p$
	as holomorphic functions of the parameter $\lambda$.
	Notice that this is the case in particular for 
	the polynomial skew products of degree 2,
	whose critical points are of the form $(z,0)$ (and so independent from $\lambda$).
	 Set 
	$$h_m(\lambda):=f_\lambda^{n_m}(y_m,c_m(\lambda))$$
	where $f_{\lambda_1}^{n_m}(y_m,c_m(\lambda_1))=(z_m, w_m)$, and $(y_m, c_m(\lambda))$ is a critical 
	point of $f_\lambda$. By definition
	we have $(\lambda, h_m(\lambda))$
is in the postcritical set.
	We write 
	$h_m(\lambda)=:(z_m, w_m(\lambda))$. 
	
	Since $(z,w) \in J_{f(\lambda_1)}$, 
	there exists 
	a sequence of repelling cycles 
	of the form $(z_m, \gamma_m(\lambda_1)$ converging to $(z,w)$
	 (by the lower semi-continuity of $z \mapsto J_z$ and the density of 
	repelling cycles).
	Since $f_\lam$ is stable, repelling cycles can be followed holomorphically. We denote by 
	$(z_m,\gamma_m(\lambda))$ the motion of $(z_m, \gamma(\lam_1))$.
	Again by the stability of the family, since there are no Misiurewicz parameters, we must have $\gamma_m(\lambda) \neq w_m(\lambda)$
	for all $m$ and for all $\lambda$. Since the sequence $\gamma_m$ is uniformly bounded, it is normal and we can assume that $\gamma_m$ converges to some
	holomorphic map $\gamma$ with $\gamma(\lambda_1)=w$.
	
	\begin{claim}\label{claim_in_stabhyp}
	The sequence $(w_m (\lam))$ is also normal.
	\end{claim}
	
	Assuming this claim, we can get the desired contradiction by taking a limit $w(\lam)$ for the sequence $w_m(\lam)$. 
Indeed, recall that $\gamma(\lam)\neq w_m (\lam)$
for all $\lambda$ and $m$. Since $\gamma(\lambda_1)=w(\lambda_1)$, by  Hurwitz's Theorem the only possibility if that
$\gamma(\lam)\equiv w(\lam)$ for all $\lam$.
Since
$\gamma(\lambda_0) \neq w(\lambda_0)$ by assumption, this gives the desired contradiction.

	\begin{proof}[Proof of Claim \ref{claim_in_stabhyp}]
	Since the family is stable, $w_m(\lambda)$ avoids the repelling cycles
	for all $m$ and $\lambda$.
	 Let $a_m(\lambda), b_m(\lambda)$ 
	be two sequences of (holomorphic motions of)
	repelling periodic points in the fibre $z_m$.
	Up to passing to subsequences, 
we can
	 assume that $a_m (\lam)\to a(\lam)$ and $b_m (\lam)\to b(\lam)$ (as holomorphic functions in $\lambda$). 
	 We can 
	also assume that $|a_m(\lambda)-b_m(\lambda)| \geq \epsilon_0>0$ for all $m$ and $\lambda$.
	Then, for all $\lam$, we have
	 $w_m(\lambda) \notin \{ a_m(\lambda), b_m(\lambda), \infty \}$. It follows that the family
	$g_m(\lambda):=\frac{w_m(\lambda)-a_m(\lambda)}{b_m(\lambda)-a_m(\lambda)}$ 
	avoids $0,1,\infty$, hence is normal by Montel Theorem and converges, up to extraction, to some $g(\lam)$.
	Since $|a_m(\lambda) - b_m(\lambda)| \geq \epsilon_0$, the sequence
	$w_m (\lam) = a_m(\lam) + g_m(\lam)\cdot (b_m(\lam)-a_m(\lam))$ 
	converges 
	to  $w(\lam) := a(\lam) +g(\lam)\cdot (b(\lam)-a(\lam))$. The claim is proved.
	\end{proof}

	\medskip
	
	We now explain how to adapt the above arguments 
	in the case where it is not possible 
	to follow all critical points as holomorphic functions of $\lambda$. As before, we 
	start with sequences  of integers $n_m$ and points $y_m \in J_p$ such that $f_{\lambda_1}^{n_m}(y_m, c_m)$
	accumulates to some point in $J_2(f_{\lambda_1})$, and $c_m$ is a critical point of $q_{y_m,\lambda_1}$. 
	The accumulation point in $J_2(f_{\lambda_1})$ can 
	also be accumulated by repelling periodic points 
	$(z_m,\gamma_m(\lambda_1))$.
	
	We now define the function
	$$h_m(\lambda):=\prod_{c_i} (f^{n_m}(y_m, c_i) - \gamma_m(\lambda)),$$
	where the product is taken over 
	the set of critical points $c_i$ of $q_{y_m,\lambda}$ whose orbits
	are bounded. Observe now that the function $h_m$ is holomorphic. 
	Indeed, \emph{for a fixed $m$}, it is always possible to mark the critical points of
	$q_{y_m,t}$ as holomorphic functions $c_i(t)$, up to
	passing to a re-parametrization $\phi(t)=\lambda$, where $\phi$ is a finite branched cover.
	
	Since the family is stable, each critical point $c_i(t)$ either has bounded orbit for all $t$ or unbounded orbit 
	for all $t$. Therefore, $t \mapsto h_m \circ \phi(t)$ is holomorphic, and since $\lambda \mapsto h_m(\lambda)$ is continuous and holomorphic outside the branch locus of $\phi$, it is also holomorphic on the whole
	family. Moreover, the sequence $(h_m)$ is locally uniformly bounded in $\lambda$, hence normal; 
	and for all $m$ and $\lambda$ we must have $h_m(\lambda) \neq 0$ since otherwise this would create  a 
	Misiurewicz parameter, contradicting the stability of the family. 
	From there, the proof works as in the previous case.
\end{proof}

For families of polynomial skew products,
it thus makes sense to talk about \emph{hyperbolic components} (respectively \emph{vertically expanding components}), i.e., stable components whose elements are (all)
hyperbolic (respectively, vertically expanding). 
We will characterize and classify 
some components of this kind
in the next sections.
An ingredient in our classification is given by the following result.

\begin{lemma}\label{lemma_extension_hol_motion}
	Let $f_{\lam}(z,w)=(p(z),q_{\lam}(z,w))$
	be a family of polynomial skew products defined on some parameter
	space $M$.
	Assume that $f_{\lam_0}$
	is uniformly vertically expanding above $J_p$.
	Then for a small enough neighbourhood $U$ of $\lam_0$ in $M$, there exists a unique continuous map
	$h: U \times J_2(f_{\lam_0}) \to \C^2$, such that:
	\begin{enumerate}
		\item for all $\lam \in U$, $h_\lam:=h(\lam, \cdot) : J_2(f_{\lam_0}) \to J_2(f_\lam)$ is a homeomorphism
		conjugating the dynamics, and
		\item $h_\lam$ is of the form $h_\lam(z,w)=(z, g_\lam(z,w))$.
	\end{enumerate}
\end{lemma}

\begin{proof}
	We follow the classical one dimensional construction
	of the conjugation
	valid for hyperbolic polynomial maps, see e.g., \cite{buffhubbard}.
	For ease of notation, we write $f_0$ for
	$f_{\lam_0}$ and assume that $\lam\in \D$.
	
	By uniform expansiveness and continuity, there exist $\eps$ and $C>1$
	such that, for every $\lam$ sufficiently small
	and every $(z,w)\in J_{2}(f_0)$ we have
	$\abs{q_{\lam,z}'(w')}> C>1$ for every $w'\in B(w,\eps)$.
	This implies that, denoting by $(z_n,w_n)$ the orbit of $(z,w)$
	under
	$f_0$, we have
	$q_{\lam,z} (B(w_n,\eps))\supset B(w_{n+1}, C'\eps)$ for some $1<C'<C$.
	It follows that the diameter of $B(w_{n+1}, \eps)$ inside $q_{\lam,z} (B(w_n,\eps))$
	is uniformly bounded from above and that, if $x,y \in B (w_n,\eps)$ and $q_{\lam,z} (x), q_{\lam,z}(y) \in B(w_{n+1},\eps)$, then
	\[
	d_{B (w_{n+1},\eps)} (q_{\lam,z}(x), q_{\lam,z}(y)) > C'' d_{B (w_n, \eps)} (x,y).
	\]
	for some uniform constant $C''>1$. Thus, the intersection
	\[B(w,\eps) \cap q_{\lam,z}^{-1} (B(w_1,\eps)) \cap \dots \cap q_{\lam, z_{n-1}}^{-1} \circ \dots \circ q_{\lam,z}^{-1} (B(w_n,\eps))\]
	consists of a single point.  Denote it by $g_\lam(z,w)$. Then, $q_{z,\lam} \circ g_\lam(z,w) = g_\lam(z_1, q_{z,0} (w))$.
	
	Let us prove that the map $g$ constructed above is continuous at $(\lam_0,z_0,w_0)$. As proved above, for $\eps>0$ small enough a basis of open neighbourhoods
	of $g_{\lam_0}(z_0,w_0)$ is given by the intersections
	$V_n(\eps):=\bigcap_{i=0}^n \pa{Q^n_{z_0,\lam_0}}^{-1} (B(w_n, \eps))$. Let 
	$$U_n(\eps):=\{(\lam,z,w) \in \D \times \C^2 : \forall k \leq n,  \left|Q_{z,\lam}^n(w)-Q_{z_0,\lam_0}^n(w_0) \right| < \eps \}.$$
	Then $U_n(\eps)$ is an open neighbourhood of $(\lam_0,z_0,w_0)$, and for all $(\lam,z,w) \in U_n(\eps)$,
	$g_\lam(z,w) \in V_n(2\eps)$. This proves the continuity of $g$. 
	
	Now we set $h_\lam(z,w):=(z, g_\lam(z,w))$. Since we can start the construction at a different $\lam$ near 0, the map $h_\lam$
	is invertible and thus a homeomorphism. Finally, to prove the uniqueness of $h$, just note that for any $\lam \in U$, $h_\lam$ must map periodic points of $f_{0}$ to periodic points of $f_\lam$ of same period;
	since periodic points of a given period are discrete, the values of $h_\lam$ are uniquely defined on periodic points, and so uniquely defined by density.
\end{proof}

\begin{cor}\label{cor:isotopy}
	Let $f_0, f_1$ be two polynomial 
	skew products in the same vertically expanding component. Then there exists an isotopy $(h_t)_{t \in [0,1]}$ in $J_p \times \C$ between $J_2(f_0)$ and $J_2(f_1)$, fixing 
	each vertical fiber.
\end{cor}

\begin{proof}
	Since $f_0$ and $f_1$ are in the same vertically expanding component, there is a continuous path $(f_t)_{t \in [0,1]}$ joining them, such that $f_t$ is vertically expanding for all $t \in [0,1]$.
	Covering the path $(f_t)$ with
	 finitely many small enough balls, we can apply Lemma \ref{lemma_extension_hol_motion} to find the required isotopy $h_t$ (the uniqueness in Lemma \ref{lemma_extension_hol_motion} makes it possible to glue each piece of the isotopy in a coherent manner).
\end{proof}

\subsection{Unbounded hyperbolic components in $\dcal$}
\label{section_hyp_D}

We establish here the second part of Theorem \ref{teo_new_hyp}.  
We work with the family $\skp$ defined as in Section \ref{section_quadratic}. Recall that this means working with
the family
$f_{\lam = (a,b,c)} (z,w)= (p(z), w^2 + az^2 + bz+c)$
for some fixed polynomial $p$ of degree $2$.

\begin{prop}\label{prop:defomega}
	Let $[\lam_i]:=[a_i,b_i,c_i] \in \ptwo_\infty \backslash E$, $i \in \{0,1\}$, and for every bounded Fatou component $U$ of 
	$p$, let $s_i(U)$ denote the number of roots of $a_i X^2+b_i X+c_i$ in $U$. If for every $U$  we have $s_0(U)=s_1(U)$
	then  both $[\lam_i]$ are accumulated by the same connected component of $\mathcal{D}$.
\end{prop}

\begin{proof}

	Theorem \ref{teo_acc_infinity_bif} implies that if $[\lam_0]$ and 
	$[\lam_1]$ are in the same connected component of $\ptwo_\infty \backslash E$, then they are accumulated 
	by the same connected component of $\Dd$. This can be seen by picking a path $t \mapsto [\lam(t)]$
	in $\ptwo_\infty \backslash E$ joining $[\lam_0]$ and $[\lam_1]$, and lifting it to a path 
	$t \mapsto \lam(t) \in \C^3$. Then, for all $n \in \N$  large enough, $n \lam(t) \in \Dd$ for all $t \in [0,1]$, 
	so $n \lam(0)$ and $n \lam(1)$ are in the same component of $\Dd$. Moreover, as $n \to +\infty$, $n \lam(t) \to [\lam(t)] \in \ptwo_\infty$. 
	
	Therefore, it remains to see that if $s_0(U)=s_1(U)$ for all bounded Fatou component $U$ of $p$, then $[\lam_0]$ and $[\lam_1]$ 
	belong to the same connected component of $\ptwo_\infty \backslash E$. Since $E$ is closed, we may 
	slightly perturb $[\lam_0]$ and $[\lam_1]$ if necessary to 
	assume 
	  that $a_i \neq 0$, and so choose representatives of the form $(1,b_i,c_i)$, so that both polynomials $X^2+b_i X+c_i$ 
	have two roots $x_i, y_i$ counted with multiplicity. By assumption, we may assume that $x_0$ and $x_1$ 
	(respectively $y_0$ and $y_1$) belong to the same Fatou components of $p$. Choosing paths $x(t)$ (respectively $y(t)$) joining $x_0$ and $x_1$ (respectively $y_0$ and $y_1$) inside those Fatou components, the path $t \mapsto [1, -x(t)-y(t), x(t) \cdot y(t)]$ joins $[\lam_0]$ and $[\lam_1]$ inside 
	$\ptwo_\infty \backslash E$.
	This concludes the proof.
\end{proof}

Let $\pi_0(\mathring{K_p})$ denote the set of all bounded Fatou components of $p$, and  set
\begin{equation}
	\mathcal{S}_p = \Big\{s: \pi_0(\mathring{K_p}) \to \N :  \sum_{U \in \pi_0(\mathring{K_p})} s(U) \leq 2 \Big\}.
\end{equation}

To any $[\lam]=[a,b,c] \in \ptwo_\infty \backslash E$, we can associate an element $s \in \mathcal{S}_p$ 
defined as follows: $s(U)$ is the number of roots of $aX^2+bX+c$ in $U$. 
It is easy to check that all $s \in \mathcal{S}_p$ can be realized in that way.
Proposition \ref{prop:defomega} asserts that to any such $s$ is 
associated a unique hyperbolic component 
of $\Dd'$: in other words, we have defined a map $\omega : \mathcal{S}_p \to \pi_0(\Dd')$,
where $\pi_0(\Dd')$ is the set of hyperbolic components of $\Dd'$.

Our result below completes the classification of these components, for $p$ with locally connected Julia set. It also completes
the proof of Theorem \ref{teo_new_hyp}. Notice that the assumption that 
  $J_p$ is locally connected implies that
the boundary of every bounded
Fatou component of $p$ is a Jordan curve, see \cite{notesorsay2}. This assumption is automatically satisfied if $p$ is hyperbolic.
Recall further that all bounded Fatou components of $p$ are simply connected.

\begin{teo}\label{teo_classification_components}
Assume that $J_p$ is locally connected. Then $\omega : \mathcal{S}_p \to \pi_0(\Dd')$ is 
bijective.
\end{teo}

Since $\omega$ is  surjective, all it remains is to prove that it is injective.
The rest of the section is devoted to that task.  We
will need the following definition.

\begin{defi}
Given $\lam \in \mathcal D$ set $r(f_\lam):=\inf_{z \in J_p} |az^2+bz+c|$.
	Let $\gamma\colon [0,1]\to \C^2$ be a simple closed curve,
		given by $\gamma(t)=(\gamma_z(t), \gamma_w(t))$. We say that $\gamma$ is admissible (for $f_\lam$)
	if for all $t \in [0,1]$, $|\gamma_w (t)| < r(f_\lam)$.
\end{defi}

\begin{lem}\label{lem:admissible}
	Let $[a_0,b_0,c_0] \in \ptwo_\infty$ be such that the roots of $a_0 X^2+b_0 X+c_0$
	 are in the Fatou set of $p$.
	Let   $\lam=(a,b,c) \in \C^3$ be such that $[a,b,c]= [a_0,b_0,c_0]$. If
	 $|\lam|$ is large enough,
	the map $f_\lam$
	satisfies the following properties: 
	\begin{enumerate}
	\item if $C$ is a curve such that $C \subset K(f_\lam)$, then 
$C$ is admissible;
		\item\label{item-admissible-preim} If $C$ is an admissible curve, then so is every component of $f_\lam^{-1}(C)$;
		\item\label{item-admissible-radius} There exists $0<r^*(f_\lam)<r(f_\lam)$ such that for all $z \in J_p$, $K_z \subset \D(0,r^*(f_\lam))$.
	\end{enumerate}
\end{lem}

\begin{proof}
	Set $R(f_\lam):=\sup_{z \in J_p} | az^2 +bz +c|$. Then 
	there exists a positive constant $\alpha=\alpha(a_0,b_0,c_0,p)$
	 such that
	 $\frac{1}{\alpha} |\lam| \leq r(f_\lam) \leq R(f_\lam) \leq \alpha |\lam|$.
	 The first item then follows from Lemma \ref{lem:adherenceinclus}.
	Moreover $\lam \in \Dd$
	for $|\lam|$ large enough
	 and for all $z \in J_p$, 
	we have $K_z \subset \D(0,2\sqrt{R(f_\lam)})$, again by Lemma \ref{lem:adherenceinclus}.
	Therefore we may take 
	$r^*:=2\sqrt{R(f_\lam)}$
	 for item \eqref{item-admissible-radius}. 
	For item \eqref{item-admissible-preim}, 
	observe that if $(z,w) \in F^{-1}(C)$ and $C$ is admissible, then 
	$|w| = O(\sqrt{|\lam|})$ and therefore any component of $F^{-1}(C)$ is also admissible.
\end{proof}

In the following, we fix a pair  $U,V$ of bounded Fatou components
 of $p$ with $p(U)=V$. Our assumption implies
that $\partial U$ and $\partial V$
are Jordan curves. 
We denote by
$s$
the number of roots of $aX^2+bX+c$ lying in $U$, counted with multiplicity.
Given $\lam$ and $C$ 
 a simple closed curve in $\partial V \times \C$, we will set
$\hat C:=f_\lam^{-1}(C) \cap (\partial U \times \C)$.

\begin{defi}
	Let $C_0 \subset \C$  and $\tilde C_0 \subset C_0 \times \C$ be two topological circles. 
	We will say that $\tilde C_0$ \emph{winds $n$ times above $C_0$} 
	if the projection $\pi_1 : \tilde C_0 \to C_0$ is an 
	unbranched covering of degree $n$.
\end{defi}

\begin{lem}\label{lem:pullbackcurve1}
Assume that $\lam=(a,b,c)\in \Dd$, 
	 the roots of $aX^2+bX+c$ are in the Fatou set of $p$ and that 
$|\lam|$
	 is large enough so
	that Lemma \ref{lem:admissible} holds.
	Assume that
	$C$ winds once above $\partial V$ and is admissible.
	Then
	\begin{enumerate}
		\item if $s=0$ or $s=2$, $\hat C$ has two connected components $C_1$ and $C_2$, and their 
		linking number in $\partial U \times \C$ 
		is equal to $s/2$. Both components wind once above $\partial U$;
		\item if $s=1$, then $\hat C$ is connected and winds twice above $\partial U$.
	\end{enumerate}
\end{lem}

\begin{proof}
	Let $\delta \in \{1,2\}$ be the degree of $p: U \to V$ ($\delta=1$ if $U$ contains no critical point of $p$,
	and $\delta=2$ otherwise). Let $\gamma : \R/\Z \to C$ defined by $\gamma(t):=(\gamma_V(t),\gamma_w(t))$
	be
	a parametrization of $C$.
	Let $\gamma_1 : \R \to \C^2$ be a lift by $F$ of $t \mapsto \gamma(\delta t)$.
We can define a parametrization $\gamma_U$
of $\partial U$ by	
	$p \circ \gamma_U(t)=\gamma_V(\delta t)$ for all $t \in \R/\Z$. So,
	the map $\gamma_1$ is of the form 
	$$\gamma_1(t)=(\gamma_U(t), w_t)$$
	and $w_t$ satisfies the equation
	$$w_t^2=\gamma_w(t)-(a \gamma_U(t)^2 +b\gamma_U(t)+c).$$
	Observe that the curve $t \mapsto \gamma_w(t)-(a\gamma_U(t)^2+b\gamma_U(t)+c)$
	turns $s$ times around $w=0$.
	We now distinguish between the cases $s \in \{0,2\}$ or $s=1$.
	
		(1)	
	If $s \in \{0,2\}$, since the curve $t \mapsto \gamma_w(t)-(a\gamma_U(t)^2+b\gamma_U(t)+c)$
		turns an even number of times around $w=0$ as $t$ goes from $0$ to $1$, we have $w_1=w_0$ by 
		monodromy. Therefore $\gamma_1(1)=\gamma_1(0)$, and 
		$\gamma_1(\R)$ is a closed loop winding once above $\partial U$. Since $F: \hat C \to C$ has degree $2\delta$, 
		and $F: \gamma_1(\R) \to C$ has degree $\delta$, there is a second lift $\gamma_2 : \R \to C_2$ parametrizing a 
		second connected component of $\hat C$. Moreover, $\gamma_2$ has the form 
		$$\gamma_2(t)=(\gamma_U(t),-w_t)$$
		and therefore the linking number
		in $\partial U \times \C$ of $C_1$ and $C_2$ is given by the number of turns around $w=0$ 
		of $t \mapsto w_t$ as $t$ varies from $0$ to $1$, namely $s/2$. 
	
	(2) If $s=1$: now the curve $t \mapsto \gamma_w(t) - (a \gamma_U(t)^2+b\gamma_U(t)+c)$ turns exactly
		once around $w=0$ as $t$ goes from $0$ to $1$. Therefore, by monodromy we have $w_1=-w_0$ and $w_2=w_0$.
		This means that the support of $\gamma_1(\R)$ is a curve that winds twice above $\partial U$. Moreover,
		as $\gamma_1(\R) \cap \{z=\gamma_U(0) \}= \{ (\gamma_U(0), \pm w_0 )\}$, the degree of 
		$F: \gamma_1(\R) \to C$ is $2\delta$ and therefore $\hat C = \gamma_1(\R)$.
\end{proof}

\begin{lem}\label{lem:pullbackcurve2}
	 Assume that $\lam=(a,b,c)\in \Dd$, 
	 the roots of $aX^2+bX+c$ are in the Fatou set of $p$ and that 
$|\lam|$
	 is large enough so
	that Lemma \ref{lem:admissible} holds.
	Assume that $C$ winds twice above $\partial V$.
	Then $\hat C$ has two connected components $C_1$ and $C_2$. Both are curves that wind twice above
	$\partial U$ and their linking number in $\partial U \times \C$ is equal to $s$.
\end{lem}

\begin{proof}
	The proof is similar to that of the previous Lemma.
	Let $\delta \in \{1,2\}$ be the degree of $p: U \to V$.
	Since $C$ winds twice above $\partial V$, it has a parametrization 
	$\gamma : \R/\Z \to C$ of the form $\gamma(t)=(\gamma_V(2t), \gamma_w(t))$,
	where for all $t \in \R$, $\gamma_w(t+\frac{1}{2}) \neq \gamma_w(t)$.
	As before, let $\gamma_1: \R \to \hat C$ be a lift by $F$ of $t \mapsto \gamma(\delta t)$.
	Then $\gamma_1$ has the form $$\gamma_1(t)=(\gamma_U(2t),w_t),$$
	and $t \mapsto w_t$ satisfies the equation
	$$w_t^2 = \gamma_w(t)-(a \gamma_U(2t)^2+b\gamma_U(2t)+c).$$
	Note that $a \gamma_U(1)^2+b\gamma_U(1)+c = a \gamma_U(0)^2+b \gamma_U(0)+c$ but 
	$\gamma_w(\frac{1}{2}) \neq \gamma_w(0)$, hence $w_{1/2} \neq w_0$.
	Also note that as $t$ varies from $0$ to $1$, the loop $t \mapsto  \gamma_w(t)-(a \gamma_U(2t)^2+b\gamma_U(2t)+c)$ turns $2s$ times around $w=0$. Therefore by monodromy, we 
	have $w_1=w_0$, so that $\gamma_1(1)=\gamma_1(0)$ and $\gamma_1(\R)$ is a closed loop that 
	winds twice above $\partial U$. 
	
	Again, the degree of $F: \hat C \to C$ is $2\delta$, and the degree of $F: \gamma_1(\R) \to C$
	is only $\delta$. Moreover, $w_{1/2} \neq -w_{0}$ (since $w_{1/2}^2 \neq w_0^2$), and therefore
	$\gamma_2 : \R \to C$ defined by $\gamma_2(t)=(\gamma_U(2t),-w_t)$ parametrizes a second and different 
	component of $\hat C$. For degree reasons,
	$\hat C$ is exactly equal to 
	$C_1 \cup C_2$, where $C_i$ is the support of $\gamma_i(\R)$. Each $C_i$ is a loop winding twice above $\partial U$,
	and $C_1 \cap C_2 = \emptyset$ since for all $t \in \R$, $w_t \neq 0$. Therefore the $C_i$ 
	are the connected components of $\hat C$. 
	Moreover, since $\gamma_1(t)=(\gamma_U(2t), w_t)$ and $\gamma_2(t)=(\gamma_U(2t),-w_t)$,
	the linking number of $C_1$ and $C_2$ is given by the number of times that $t \mapsto w_t$ turns around
	$w=0$ as $t$ varies from $0$ to $1$, namely $s$.
\end{proof}

On our way to prove Theorem \ref{teo_classification_components}, 
we will need the following topological description of 
the Julia sets of maps in $\dcal$, which has independent interest.

\begin{defi}\label{def:suspension}
	Let $\Sigma\subset \C$ be a Cantor set that is invariant under $w \mapsto -w$.
	The \emph{suspension} $S$ of $\Sigma$ is given by $S:= ([0,1] \times \Sigma) / \sim$,
	where  $(0,w)  \sim (1,-w)$.   
\end{defi}

\begin{teo}\label{th:topologyjulia}
		Assume that $J_p$ is locally connected, and let $U$ be a bounded Fatou component of $p$.
	Let $[a,b,c] \in \ptwo_\infty \backslash E$. For all representative $\lam=(a,b,c)$ of norm large enough,
	the following holds:
	\begin{enumerate}
		\item If there exists $n \in \N^*$ such that $p^n(U)$ contains exactly one root of $aX^2+bX+c$, 
		then $J_{\partial U}:= \cup_{z\in \partial U} \{z\}\times J_z$ is homeomorphic to $S$.
		\item  Otherwise, $J_{\partial U}$ is homeomorphic 
		to $S^1 \times \Sigma$.
	\end{enumerate}
Moreover, if $p$ has no bounded Fatou components then $\mathcal{D}'$ has only one component, in which the Julia set 
is homeomorphic to $J_p \times \Sigma$.
\end{teo}

\begin{proof}
	If $p$ has no periodic bounded Fatou component, then by Sullivan's Theorem $p$ only has 
	the basin of infinity as a Fatou component. In this case, there is only one component in 
	$\mathcal D'$. Indeed, in this case  $\mathcal{S}_p$ is a singleton and $\mathcal{D}'=\omega(\mathcal{S}_p)$.
	This component must necessarily contain product maps; therefore $J_2(F)$ is homeomorphic to $J_p \times \Sigma$.
	From now on, we assume that $p$ has a cycle of bounded Fatou components.

	Let $r^*=r^*(f_\lam)$ be given by Lemma \ref{lem:admissible}.
	Let $U_0$ be a bounded periodic Fatou component for $p$ of period $m \in \N^*$,  
	and let $W_0:=\partial U_0 \times \D(0,r^*)$. 
	Let $U_i:=p^{m-i}(U_0)$ be a cyclic numbering of the cycle of components 
	containing $U_0$, with $i=0, \dots, m-1$,  so that
	$p(U_{i+1})=U_i$.
	\begin{enumerate}
		\item Assume first that each component in the cycle $U_0, \ldots, U_{m-1}$ contains either zero or two roots of $aX^2+bX+c$.
		Since $W_0$ is homotopic to a curve winding once above $\partial U_0$, by Lemma \ref{lem:pullbackcurve1}, 
		$W_1:= F^{-1}(W_0)  \cap (\partial U_{1} \times \C)$ is homotopic to two disjoint curves, each winding once
		above $\partial U_{1}$. Therefore, $W_1$ is a disjoint union of the interior of two solid tori, each winding once
		above $\partial U_{1}$. Letting $W_n := F^{-n}(W_0) \cap (\partial U_{_n} \times \C)$, we 
		therefore get by induction that $W_n$ is a disjoint union of the interior of $2^n$ solid tori, each winding once
		above $\partial U_n$. Since $W_m \Subset W_0$, we get that $\bigcap_{ n \in m \N} W_n$ is homeomorphic 
		to $S^1 \times \Sigma$. 
		\item Assume now that there exists a component in the cycle containing $U_0$ (we may assume without loss of 
		generality that it is $U_0$ itself) that contains exactly one root of $aX^2+bX+c$.
		We proceed as before, letting $W_0:=\partial U_0 \times \D(0,r^*)$ and $W_n:=F^{-n}(W_0) \cap (\partial U_n \times \C)$. This time, Lemma \ref{lem:pullbackcurve1}  implies that $W_1$ is homotopic to a double winding curve
		above $\partial U_1$. Therefore $W_1$ is the interior of a double winding solid torus. Moreover, by Lemma \ref{lem:pullbackcurve2}, for all $n \geq 1$ the set 
		$W_n$ is the disjoint union of the interior of $2^{n-1}$ solid tori, each winding twice above $\partial U_n$.
		Therefore, for $0 \leq j \leq m-1$, $J_{\partial U_j}= \bigcap_{n \in m \N} W_n$ is homeomorphic to the suspension $S$.
	\end{enumerate}
	
	To conclude
	the proof of Theorem \ref{th:topologyjulia}, notice that if $U,V$ are two Fatou components of $p$
	such that $p(U)=V$, and $J_{\partial V}$ is homeomorphic to either $S^1 \times \Sigma$ or $S$, then 
	Lemmas \ref{lem:pullbackcurve1} and \ref{lem:pullbackcurve2} allow us to determine the topology of $J_{\partial U}$.
	More precisely, letting $s \in \{0,1,2\}$ be the number of roots of $aX^2+bX+c$ contained in $U$, we have
	the following:
	\begin{enumerate}
		\item if $s=0$ or $s=2$, then $J_{\partial U}$ is homeomorphic to $J_{\partial V}$;
		\item if $s=1$, then $J_{\partial U}$ is homeomorphic to $S$. 
	\end{enumerate}
	Since every Fatou  component of $p$ is preperiodic to $U_0$,
	the rest of the proof follows.
\end{proof}

We are now ready to prove the injectivity of $\omega$.

\begin{proof}[Proof of Theorem \ref{teo_classification_components}]
	Let $s_0,s_1 \in \mathcal{S}_p$ such that $\omega(s_0)=\omega(s_1)$: we need to prove that $s_0=s_1$.
	In other terms, let $f_0,f_1 \in \skp$ be in a small enough neighbourhood of $\ptwo_\infty \backslash E$
	in $\P^3$, 
	and belonging to the same component of $\Dd'$; we will prove that for every $U \in \pi_0(\mathring{K_p })$, $s_0(U)=s_1(U)$. 
	
	Recall that by Corollary \ref{cor:isotopy}, if $(f_t)_{t \in [0,1]}$ is an arc in $\mathcal D$, there is an isotopy $h_t : J_2(f_0) \to J_2(f_t)$ of the form $h_t(z,w)=(z, g_t(z,w))$.
	Since $f_0,f_1$
	are in the same connected component of $\mathcal D'$,  they can be joined by such an arc
	and therefore their Julia sets are isotopic in $J_p \times \C$.

	Let $U$ be a bounded Fatou component of $p$ and let $z \in \partial U$, $w \in J_z(f_0)$. 
	By Theorem \ref{th:topologyjulia},
	there exists a unique closed simple curve $C_0$ passing through $f_0(z,w)$ and contained in 
	$J_2(f_0) \cap (\partial V \times \C)$, where $V:=p(U)$. That curve winds either once or twice above $\partial V$. 
	Let $C_1:=h_1(C_0)$ and $\hat C_i:=F_i^{-1}(C_i) \cap (\partial U \times \C)$, where $i \in \{0,1\}$.
	Since the number of connected components of $\hat C_i$ and their linking number 
	in $\partial U \times \C$
	are invariant under isotopy in $\partial U \times \C$, 
	Lemmas \ref{lem:pullbackcurve1} and \ref{lem:pullbackcurve2} imply that $s_0(U)=s_1(U)$. 
	Since this is true for any bounded Fatou component $U$ of $p$, $s_0=s_1$ and the proof is finished.	
\end{proof}

\subsection{Unbounded hyperbolic components in $\mcal$}\label{section_hyp_other}

We have provided in the previous section a complete classification of unbounded components 
of $\dcal$ accumulating on $\P^2_\infty \setminus E$. In this section
 we adapt an interesting example (\cite[Example 9.6]{jonsson1999dynamics})
 to construct unbounded hyperbolic components in $\mcal$.
For the sake of notation, we start setting the following definition, motivated by Corollary
\ref{coro:3bzbounded}.

\begin{defi}
	Let $p$ be a quadratic polynomial.
	Let $z_1, z_2 \in J_p$ (possibly with $z_1=z_2$). We say that a hyperbolic component
	$U \subset \skp$ is \emph{of type $\{z_1,z_2\}$} if for all $z \in J_p$, $G(z,0)=0$ if and only if 
	$z=z_1$ or $z=z_2$. We may write $\{z_1\}$ instead of $\{z_1,z_1\}$.
\end{defi}

The following theorem provides
a basic classification of unbounded hyperbolic components in $\mcal$.
While for components of $\Dd$ we looked for a correspondence
with (pairs of) points in the Fatou set of $p$, for unbounded components of 
$\mcal$
we see that a natural correspondence
exists with (pairs of)
points in the Julia set of $p$.

\begin{teo}
	Let $p$ be a quadratic polynomial and $U \subset \skp$ be an unbounded hyperbolic component
in $\mcal$.
	Then there are $z_1, z_2 \in J_p$ such that $U$ is either of type $\{z_1\}$ or of type 
	$\{z_1,z_2\}$. Moreover, if $U$ is of type $\{z_1\}$ then $z_1$ must be periodic for $p$, 
	and if it is 
	of type $\{z_1,z_2\}$ then either both $z_1$ and $z_2$ are periodic or one is preperiodic to the 
	other.
\end{teo}

\begin{proof}
	By Theorem \ref{teo_new_skew_things},
	for any $f_1, f_2 \in U$ and $z \in J_p$, we 
	have that $(z,0)$ has a bounded orbit for $f_1$ if and only if it has a bounded orbit for $f_2$.
	Since $U$ is unbounded, Corollary \ref{coro:3bzbounded} implies that there are at most 
	two points $z_1, z_2 \in J_p$ such that $(z_i,0)$ has bounded orbit, and since $U$ is a component in 
	$\mcal$ there is at least one $z \in J_p$ such that $(z,0)$ has bounded orbit. Therefore there are 
	$z_1, z_2 \in J_p$ (possibly with $z_1=z_2$) such that $U$ is of type $\{z_1,z_2\}$.
		In order to prove the remaining claims of the theorem, we will use the following lemma.

	\begin{lem}\label{lem:visitcritpoint}
		Let $f$ be a polynomial skew product that is vertically expanding above $J_p$. Let $z \in J_p$ and 
		$V$ be a connected component of $\mathring{K_z}$. There exists $n \in \N^*$ such that 
		$f^n ( \{z\}\times V)$ contains a critical point for $f$.
	\end{lem}

	We refer to \cite[Proposition 3.8]{demarco2008axiom} for
	a proof of this fact. It is stated there
	in the case of an Axiom A polynomial skew product
	but the proof only uses vertical expansion over $J_p$.

	Assume first that $U$ is of type $\{z\}$, and let $V$ be the connected component of 
	$\mathring{K_z}$ containing  $0$.
	 By Lemma \ref{lem:visitcritpoint}, there is $n \in \N^*$
	such that $f^n(\{z\}\times V)$ contains a critical point for $f$. 
	But since all critical points $(y,0)$, $y \in J_p$ 
	escape if $y \neq z$, this means that $f^n(\{0\}\times V)=\{0\}\times V$
	 and $(z,0) \in V$. In particular, we must have
	$p^n(z)=z$.
	Similarly, if $U$ is of type $\{z_1,z_2\}$, let $V_i$ denote the component of $\mathring{K_{z_i}}$
	containing $0$
	 ($1 \leq i \leq 2$). By Lemma \ref{lem:visitcritpoint}, there are 
	$n_1, n_2 \in \N^*$ such that $f^{n_i}(\{0\}\times V_i)$ is either $\{0\}\times V_1$ or $\{0\}\times V_2$. The result 
	follows. 
\end{proof}

We now 
 give examples of all three possibilities of unbounded hyperbolic components in $\mcal$.
We need the following
elementary lemma, following
from Section \ref{section_quadratic}.

\begin{lem}
	Let $z_1, z_2 \in J_p$ with $z_1 \neq z_2$ and assume that $U$ is a hyperbolic unbounded component
	of type $\{ z_1,z_2\}$. Then the cluster of $U$ on $\ptwo_\infty$ is exactly 
	$\{[1, -z_1-z_2, z_1 z_2] \}$.
\end{lem}

The following is an adaptation of \cite[Example 9.6]{jonsson1999dynamics}.

\begin{prop}\label{prop:exjonsson}
	Let $p(z):=z^2-2$, and
	let $g_t(z,w):=(p(z), w^2+t(z+1)(2-z) )$.
	Then for all $t >0$ large enough,
	\begin{enumerate}
		\item $g_t$ is hyperbolic;
		\item for all $z \in J_p \backslash \{-1,2\}$, the critical point $(z,0)$ escapes to infinity;
		\item the critical points $(-1,0)$ and $(2,0)$ are fixed.
	\end{enumerate}
\end{prop}

\begin{proof}
 	Observe that for all $z \in J_p$ and $t$ large enough, 
 	$R:=3\sqrt{t}$ is an escape radius
	(i.e., $K_z \subset \D(0,3\sqrt{t})$ and $|w|\geq 3\sqrt{t}$ implies that $|Q_z(w)|>|w|\geq 3\sqrt{t}$).
	Set
	$$A_t:=\{ z \in [-2,2] : t (z+1)(2-z) \geq 3 \sqrt{t}\} \subset (-1,2).$$ 
	
	\begin{claim}\label{claim1}
		For $t>0$ large enough, 
		for any $z \in J_p \backslash \{-2,-1,2\}$ there exists $n \geq 0$ such that $p^n(z) \in A_t$.
	\end{claim}
	
	\begin{proof}[Proof of Claim \ref{claim1}]
		Notice
		that $p$ is semi-conjugated on $J_p$ to the doubling map on $\R /\Z$ via the 
		map $\phi: \R /\Z \to J_p$ given by $\phi(x)=2\cos(2\pi x)$. Note that $\phi([0])=2$ and 
		$\phi([\frac{1}{3}])=\phi([\frac{2}{3}])=-1$.
Given 		$0<\epsilon<1/8$, let us set
		$$\tilde A_\eps := \left(\epsilon, \frac{1}{3}-\epsilon\right) \cup \left(\frac{2}{3}+\epsilon,1-\epsilon\right)
		 \subset \R/\Z.$$
We claim that		
		for any $\theta \in \R/\Z \backslash \{ [0],[\frac{1}{3}],[\frac{1}{2}],[\frac{2}{3}] \}$, there exists $n \in \N$ such that
		$2^n \theta \in \tilde A_\eps$.
		
		\begin{enumerate}
\item If $\theta \in [-\epsilon,\epsilon]$ and $\theta \neq 0$
		then $2^n \theta \in \tilde A_\eps$ for some $n$ sufficiently large;
		\item If $\theta \in I_\eps:=[\frac{1}{3}-\epsilon, \frac{2}{3}+\epsilon]$, because of (1), we can assume by contradiction that
		$2^n \theta \in I_\eps \cup \{[0]\}$ for all $n$. This would imply that
		$2^{n+1}\theta$ belongs to $I$ for all $n$, and so necessarily to a
		 small neighbourhood of $\{ \frac{1}{3},\frac{2}{3}\}$. The only possibility
		 is that $\theta \in \{ \frac{1}{3}, \frac{2}{3} \}$, which gives the desired contradiction.
		\end{enumerate}
		
		By conjugating with $\phi$, it follows from the above 
		that, for any $\delta>0$ small enough, 
		 for any $z \in J_p \backslash \{-2,-1,2\}$
		there exists $n \in \N$
		such that $p^n(z) \in (-1+\delta, 2-\delta)$. Since for $t>0$ large enough 
		we have $(-1+\delta, 2-\delta)\subset A_t$, the assertion follows.
	\end{proof}

	\begin{claim}\label{claim2}
	Set $U_\delta:=\{ |\mathrm{Im} (w)| \leq \delta, |\mathrm{Re}(w)| \leq \frac{1}{3} \}$
and $U'_\delta:=\{ |\mathrm{Im} (w)| \leq \delta, |\mathrm{Re}(w)| \leq \frac{1}{4} \}$.
		For all $t>0$ large enough, for any $\delta_0>0$ small enough, there exists $\delta_1,\delta_2<\delta_0$ such that the following hold:
	\begin{enumerate}
			\item 
			$U_{\delta_2}\cap K_z = \emptyset$ for all $z \in J_p$ such that
			$\min ( |z+1|, |z-2|)>\delta_1$;
			\item 
			for all $z \in J_p$ such that $\min (|z+1|, |z-2|)\leq\delta_1$ 
			 we have 
			$q_z(U_{\delta_2}) \subset U'_{\delta_2}$;
			\item for all $z \in J_p \backslash \{-1,2\}$, 
			we have $U_{\delta_2} \cap K_z = \emptyset$.
		\end{enumerate}
	\end{claim}

	\begin{proof}[Proof of Claim \ref{claim2}]
		Let us prove each item separately.
		\begin{enumerate}
			\item Since $K$ is closed, it is enough to prove that for all $z \in J_p \backslash \{-1,2\}$,
$U_{\delta_2}\cap K_z=\emptyset$.		
			 First fix $z \in J_p \backslash \{-1,2\}$ and $w \in \R$. For $t$ large enough, by Claim \ref{claim1}
			 there is some $n \geq 0$ such that 
						$p^n(z) \in A_t$.  Set $w_n:=Q_z^n(w) \in \R$. Then $Q_z^{n+1}(w)=w_n^2+t(z+1)(2-z) \geq t(2-z)(z+1)\geq 3\sqrt{t}$ and therefore
						$f^n(z,w) \notin K$, hence $(z,w) \notin K$, as desired. 	
Let us now take $z=-2$ and $w \in U_{\delta_2}$, so that $|w^2|<\delta^2_2 + 1/9$. In this case, we have
$|f(-2, w)|= |w^2- 4t| \geq 4|t|- \delta^2_2 - 1/9 $, which is larger than the escape radius. The  proof of the first item 
is complete.

			\item Let $t>0$ be large enough for Claim \ref{claim1} to hold, and let $\delta_1, \delta_2>0$ be given by
			the previous item.
Fix $z \in J_p$
as in the statement. Note that 
			 $\min(|z-2|,|z+1|)\leq \delta_1$, 
			implies 
			 $|t(z+1)(z-2)| \leq 4t \delta_1$.
			Taking
$w \in U_{\delta_2}$ and setting $w_1 := q_z (w)$, we have		 				
					\[
			\begin{cases}
			\re(w_{1}) = \re(w)^2-\im(w)^2 + t (2-z)(z+1)\\		
			\im(w_{1})= 2 \im(w) \re(w)
			\end{cases}
			\]
			and therefore
			\[
			\begin{cases}
			|\re(w_{1}) |\leq \frac{1}{9} + \delta_2^2 +4 t \delta_1 < \frac{1}{4}\\
			|\im(w_{1})|\leq 2 \delta_2 \frac{1}{3} \leq \delta_2
			\end{cases}
			\]
			provided that $\delta_1$ and $\delta_2$ are small enough. The assertion follows.
			\item 
			Because of item (1), we only need to consider $z$ such that $0<\min(|z-2|,|z+1|)\leq \delta_1$.
			For any such $z$
			 and $w\in U_{\delta_2}$,
			by means of Claim \ref{claim1} and
			iterating the second item
			we find a smallest $n \geq 1$ such that
	$p^n (z)\in A_t$		
			and $Q^n_z (w)\in U_{\delta_2}$. 
			By the first item, 
			$f^n(z,w) \notin K$; so $(z,w) \notin K$, and the proof is complete.
		\end{enumerate}
	\end{proof}

Let us now return to the proof of Proposition \ref{prop:exjonsson}. 
Item 3 is trivial.
	Item 2  follows immediately from the last item of Claim 
	\ref{claim2}. In order to prove that $g_t$ is indeed hyperbolic, 
	we apply Theorem \ref{teo_jonsson_unique} and prove that the post-critical set does not accumulate on
	the Julia set $J$.
	Since the critical set
	over $J_p$ is given by $[-2,2]\times \{0\}$, it is enough to prove that
	\begin{center}
	for every $z\in [-2,2]$, we have  $d( g_t^n (z,0), J) >\delta_2$ for every $n\geq 0$.
	\end{center}
	where $\delta_2$ is as in Claim \ref{claim2}. We can assume that $\delta_2<\frac{1}{12}$ and that the distance
	between $J$ and $J_p \times \{|w|\geq 3\sqrt t \}$ is also larger than $\delta_2$.

Item 3 of Claim \ref{claim2} and the lower semicontinuity of $z\mapsto J_z$
imply that
$J \cap \pa{[-2,2]\times U_{\delta_2} } =\emptyset$. 	
Thus, the claim is true for $n=0$.	
	Since $(2,0)$ and $(-1,0)$ are fixed, the claim is true for these two points.
Moreover, the claim holds for every $z\in A_t$, since by definition
$q_z (0) = t (z+1)(z-2)\geq 3 \sqrt{t}$.
	
Fix any other $-2\neq z\in J_p$ and
set $(z_n,w_n ) := ( p^n (z), Q_z^n(0))$. 
Notice that $w_n \in \R$.	
	By Claim \ref{claim1}, there exists
$n$ such that $z_n \in A_t$.  By the first item of Claim \ref{claim2}, it
is then enough to prove that $d((z_j, w_j), J)\geq \delta_2$ for
$1\leq j < n$. But the second item of Claim \ref{claim2} implies that
$w_j \in \R \cap U'_{\delta_2}$. Since $J \cap \pa{ [-2,2] \times U_{\delta_2}} = \emptyset$, the assertion follows.

To conclude the proof, we need to consider the orbit of $(-2,0)$. But $|f^n (-2,0)|> 3\sqrt{|t|}$ for all $n\geq 1$,
as proved in the first item
of Claim \ref{claim1}. The proof is complete. 
\end{proof}

	\begin{prop}\label{prop_jonsson_bis}
		Let $p(z)=z^2-2$. There are unbounded hyperbolic components in $\skp$ of type
		$\{ -1,2\}$, $\{2\}$, and $\{-2,2\}$.
	\end{prop}
	
Notice the the component of type $\{ -1,2\}$ corresponds
to the case with two periodic points for $p$, while for the component of type
$\{ -2,2\}$ the point $2$  is periodic and $-2$ is preperiodic to $2$.
	
	\begin{proof}
		According to Proposition \ref{prop:exjonsson}, the maps $g_t$
		are all hyperbolic for $t$ large enough, 
		and since $t \mapsto g_t$ is a continuous, unbounded path in $\skp$,
		they all belong to the same hyperbolic component which is unbounded and 
		of type $\{ -1,2\}$. 
		The existence of components of type $\{2\}$ and $\{-2,2\}$ can be proved
		considering 
		skew products of respective forms 
		$(z,w) \mapsto (z^2-2, w^2+t(2-z))$ and $(z,w) \mapsto (z^2-2, w^2+t(z+2)(2-z))$, and 
		adapting Proposition \ref{prop:exjonsson} to those cases. 
	\end{proof}

\appendix

\section{Equidistribution results in parameter spaces (Theorem \ref{teo_new_equidistribution})}\label{section_equidistribution}

We obtain here a general parametric 
equidistribution
result for families of endomorphisms of $\P^k$, in any dimension $k$, see Theorem \ref{th:convproduit}. We then
describe the
 adapted version for families of polynomial skew products
  that was used in Section \ref{section_current_infinity}, see Theorem \ref{th:equidistribskew}.

\subsection{Equidistributions for endomorphisms of $\pk$}

Let $M$ be a connected complex manifold, and let 
$f\colon M\times \P^k \to \P^k$
be a holomorphic map, defining a holomorphic family $f (\lam,z)= (\lam, f_\lam (z)) $ of endomorphisms of $\pk$.
We assume here the following:

\smallskip

\begin{center}
$\forall n \in \N^* \exists 
\lambda \in M$
 such that for all periodic points of exact period
$n$
for
$f_\lambda:$ 
\end{center}
\begin{equation}	\label{eq_hp}
\det (Df_\lambda^n(z)-\id ) \neq 0.
\end{equation}
Denote by $\jac$ the determinant of the Jacobian matrix and set
$$\widetilde{\per_n^J}=\{ (\lambda, \eta) \in M\times \C: \exists z \in \pk \text{ of exact period } n \text{ for }f_\lambda \text{ and  such that }
\jac f_\lambda^n(z)=\eta  \}.$$
Let $\per_n^J$ be the closure of $\widetilde{\per_n^J}$ in $M \times \C$.
The following result in particular implies that
 $\per_n^J$ is an analytic hypersurface in $M \times \C$.

\begin{prop}\label{prop:nodynatomic}
Let $(f_\lam)_{\lam\in M}$ be a holomorphic family of endomorphisms of $\P^k$ satisfying \eqref{eq_hp}.
	There exists a sequence of holomorphic maps $P_n : M \times \C \rightarrow \C$ such that
	\begin{enumerate}
		\item for all $\lambda \in M$, $P_n(\lambda, \cdot)$ is a monic polynomial of degree $\delta_n \sim \frac{d^{nk}}{n}$;
		\item $P_n(\lambda,\eta)=0$ if and only $(\lambda,\eta) \in \per_n^J$.
	\end{enumerate}
	Moreover, if $(\lambda, \eta) \in{\per_n^J} \backslash  \widetilde{\per_n^J}$, there exists 
	$z \in \pk$ and $m <n$ dividing $n$ such that
	$f_\lambda^m(z)=z$, $\jac (f_\lambda^n)(z)=\eta$, and $1$ is an eigenvalue of $Df_\lambda^n(z)$.
\end{prop}

\begin{proof}

	Let $\Omega_n$ denote the set of $\lambda \in M$ such that periodic cycles of period less than or equal to 
	$n$ 
	do not have 1 as an eigenvalue. By Assumption \eqref{eq_hp},
	 $\Omega_n$ is open and dense in $M$. 
	Let $p_n \colon \Omega_n \times \C \to \C$  be defined by
	\[p_n
	(\lambda,\eta) := \prod_{z \in E_n(\lambda)}( \eta - \jac \, f_\lambda^n(z) )\]
	where $E_n(\lambda)$ denotes the set of periodic points of exact period $n$ for $f_\lambda$.

	By the implicit function theorem and the definition of $\Omega_n$,
	 $p_n$ is holomorphic on $\Omega_n \times \C$.
	Since it is locally bounded, Riemann's extension theorem implies
	that it extends holomorphically to
	 all of $M \times \C$.
	
	Now notice that for all $\lambda \in \Omega_n$, $n$ divides the multiplicity of 
	every root $w$ of the polynomial $p_n(\lambda, \cdot)$. Indeed, if 
	$z \in E_n(\lambda)$ is such that $w = \jac \, f_\lambda^n(z)$, then it is also the case 
	for the other points of the cycle, namely the $f^m(z)$, $0 \leq m \leq n-1$.
	So, for every $\lambda \in \Omega_n$, there is a unique monic polynomial map $P_n(\lambda, \cdot)$ such that 
	$P_n(\lambda, \cdot)^n = p_n(\lambda, \cdot)$. Its degree $\delta_n$ satisfies $\delta_n \sim \frac{\card E_n(\lambda)}{n}$, and by classical computations $\card E_n(\lambda) \sim d^{nk}$. 
	 The map $\lambda \mapsto P_n(\lambda, \cdot)$
	is holomorphic on $\Omega_n$ and locally bounded, hence extends holomorphically to $M$.

	Finally, for all $(\lambda,\eta) \in \Omega_n \times \C$, 
	$P_n(\lambda, \eta)=0$ if and only if  $\lambda \in \tilde{\per_n^J}$.
	If $\lambda \notin \Omega_n$, by considering a sequence
	$(\lambda_i,\eta_i) \in \Omega_n \times \C$ converging to $(\lambda, \eta)$, 
	we find that $P_n(\lambda,\eta)=0$ if and only if $f_\lambda$ has a cycle 
	with Jacobian $\eta$
	whose 
	period divides $n$.
	The drop in period may occur if two points of the cycle collide, creating an eigenvalue 1.
\end{proof}

\begin{defi}
For $\eta \in \C$, we denote by $\per_n^J(\eta)$ the analytic hypersurface of $M$ defined by
$\per_n^J(\eta):=\{\lambda \in M: (\lambda,\eta) \in \per_n^J \}$
and by
$L_n \colon M\times \C \to \C$ 
the function
 $L_n (\lambda,\eta) :=
d^{-nk}   \log \left| P_n(\lambda,\eta) \right| .$
\end{defi}

By the Lelong-Poincaré equation, we have that 
$dd_{\lambda,\eta}^c L_n = d^{-nk} [\per_n^J]$,
where
$[\per_n^J]$
 is the
  current of integration on $\per_n^J$.
Likewise, we have 
$dd_\lambda^c L_n(\cdot,\eta) =
d^{-nk}
 [\per_n^J(\eta)].$

\begin{teo}\label{th:convproduit}
Let $(f_\lam)_{\lam\in M}$ be a holomorphic family of endomorphisms of $\P^k$  and 
	assume that there is at least one parameter
	$\lambda_0\in M$ such that $f_{\lambda_0}$ is Axiom A 
	and $\{ \jac(f_{\lambda_0}^n)(z): f_{\lambda_0}^n(z)=z \text{ and }n \in \N \}$ is not dense in $\C$.
Then $L_n \to L$, the convergence taking place in $\lloc(M \times \C)$.
In particular,	
for any $\eta \in \C$ outside of a polar set, we have 
	$d^{-nk} [\per_n^J(\eta)] \to \tbif.$
\end{teo}

Recall that we denote by
 $L : M \to \R^+$ the sum of the Lyapunov exponents of $f_\lambda$ with respect to its equilibrium measure 
$\mu_\lambda$.
An endomorphism
$f\colon \P^k \to \P^k$ is \emph{Axiom A}
 if periodic points are dense in $\Omega_f$, and $\Omega_f$ is hyperbolic. 
Here $\Omega_f$ denotes the \emph{non-wandering set} of $g$, i.e., 
$$\Omega_f:=\{ z \in \pk : \forall U \text{ neighbourhood of }z, \exists n \in \N^* \text{ s.t. } f^n(U) \cap U \neq \emptyset     \}.$$

A family with an Axiom A parameter satisfies Assumption \eqref{eq_hp}.
Moreover, the
assumptions of Theorem \ref{th:convproduit} are for instance satisfied when
we consider the family of all endomorphisms of $\P^k$ of a given algebraic degree. Thus Theorem \ref{th:convproduit} implies
Theorem \ref{teo_new_equidistribution}.
In order to prove the convergence in Theorem \ref{th:convproduit},
in the spirit of \cite{bassanelli2009lyapunov}
we first study the convergence of the following
 modifications of 
  $L_n$:
\begin{enumerate}
\item	$L_n^+(\lambda,\eta)  = (n d^{nk})^{-1} \sum_{z \in E_n(\lambda)} \log^+ |\eta-\eta_{n}(z,\lambda)|$ 
where 
	$\eta_n(z,\lambda):=\jac f^n_\lambda(z)$
\item $L_n^r(\lambda)  = (2 \pi d^{nk})^{-1} \int_{0}^{2\pi} \log |P_n(\lambda, r e^{it})| dt$.
\end{enumerate}

We will need the following quantitative approximation of $L$
by Berteloot-Dupont-Molino.

\begin{lem}[\cite{berteloot2008normalization}, Lemma 4.5]\label{lem:mostcyclesrepel}
	Let $f$ be an endomorphism of $\pk$ of algebraic degree $d \geq 2$. 
	Let $\epsilon>0$ and let $R_n^\epsilon(f)$ be the set of repelling periodic points $z$ of exact period $n$ for 
	$f$, such that $\left|\frac{1}{n}\log|\jac f^n(z)| - L(f)\right|\leq 2\epsilon$.
	Then for $n$ large enough, $\card R_n^\epsilon(f) \geq d^{nk} (1-\epsilon)^3$.
\end{lem}

\begin{lemma}\label{lem:L+bded}
	For all $\eta \in \C$,  $L_n^+(\cdot, \eta)\to L$ pointwise and in $\lloc (M)$.
\end{lemma}

\begin{proof}
	In what follows, the notation $O(\cdot)$ denotes quantities that are bounded by constants depending only on 
	$f_\lambda$ and $\eta$, and not on $n$ or $\epsilon$.
	Fix $\eta \in \C$
	and$\epsilon>0$. 
	We have, for all $n \in \N^*$:
	\begin{equation*}
	|L_n^+(\lambda,\eta)| \leq \frac{\card E_n(\lambda)}{d^{nk}} \sup_{z \in \pk} \| Df_\lambda(z)\|
	\end{equation*}
	which is locally
	bounded from above. Moreover,
	\[
	\begin{aligned}
	L_n^+(\lambda,\eta) = 
	& \frac{1}{nd^{nk}} \left(\sum_{z \in R_n^\epsilon(\lambda)} \log |\eta-\eta_n(z,\lambda)|
	+ \sum_{z \in E_n(\lambda) \setminus R_n^\epsilon(\lambda)} \log^+ |\eta-\eta_{n}(z,\lambda)| \right)   \\
	= &  \frac{1}{nd^{nk}}
	\pa{ \sum_{z \in R_n^\epsilon(\lambda)} \log |\eta_n(z,\lambda)| + O\left((L(\lambda)-2\epsilon)^{-n} \right)}\\
	& +O \left(\frac{\card E_n(\lambda)\setminus R_n^\epsilon(\lambda)}{n d^{nk}} \log(|\eta| + (L(\lam)+2\eps)^n) \right)  
	\end{aligned}\]
	For any $\eps>0$ small enough, $\limn  (L(\lambda)-2\epsilon)^{-n}=0$. By Lemma
 	\ref{lem:mostcyclesrepel}, for $n$ large enough, 
 	$\frac{\card E_n(\lambda)\setminus R_n^\epsilon(\lambda)}{d^{nk}} =O(\eps)$. Hence,
 	 for $n$ large enough,
 	$$L_n^+(\lambda,\eta) = \frac{1}{nd^{nk}} \sum_{z \in R_n^\epsilon(\lambda)} \log |\eta_n(z,\lambda)| + O(\epsilon) = L(\lambda)+O(\eps).$$
	Therefore the sequence of maps $L_n^+$ converges pointwise
	to $(\lambda,\eta) \mapsto L(\lambda)$
	on $M \times \C$.
	Since the $L_n$'s are psh and locally uniformly
	bounded from above, by Hartogs lemma, the convergence also happens in $\lloc$.
\end{proof}

\begin{lemma}\label{lemma_conv_lr}
	For all $r>0$, 
	 $L_n^r\to L$
	 pointwise
	and in $\lloc(M)$.
\end{lemma}

The proof 
 is a straightforward adaptation of that of \cite[Theorem 3.4 (2)]{bassanelli2009lyapunov}.

\begin{proof}[Proof of Theorem \ref{th:convproduit}]
	First, note that the sequence $L_n$ does not converge to $-\infty$.
	Indeed, by assumption there is $\eta_0 \in \C$ and $r>0$ such that no cycle 
	of $f_{\lambda_0}$ has a Jacobian in $\D(\eta_0,r)$. Moreover, since $f_{\lambda_0}$ is Axiom A, 
	its cycles move holomorphically for $\lambda$ near $\lambda_0$, which implies that 
	$(\lambda_0,\eta_0) \notin \overline{\bigcup_{n \in \N^* } \per_n}$. Therefore the sequence 
	$L_n(\lambda_0,\eta_0)$ does not converge to $-\infty$.

	Let $\phi : M \times \C \rightarrow \R$ be a psh function such that a subsequence $L_{n_j}$ converges
	$\lloc$ to $\phi$. Let $(\lambda_0,\eta_0) \in M \times \C$. We have to prove that 
	$\phi(\lambda_0,\eta_0)=L(\lam_0)$.
	
	First, let us prove that $\phi(\lambda_0,\eta_0) \leq L(\lambda_0)$.
	Take $\epsilon>0$ and let $B_\epsilon$ be the ball of radius $\epsilon$ centered at 
	$(\lambda_0,\eta_0)$ in $M \times \C$.
	Using the submean inequality and the $\lloc$ convergence of $L_n^+$, we have
	\begin{equation*}
	\phi(\lambda_0,\eta_0) \leq  \frac{1}{|B_\epsilon|} \int_{B_\epsilon} \phi
	\leq \frac{1}{|B_\epsilon|} \lim_j \int_{B_\epsilon}  L_{n_j}   
	\leq \frac{1}{|B_\epsilon|} \lim_j \int_{B_\epsilon}  L_{n_j}^+   
	\leq \frac{1}{|B_\epsilon|}  \int_{B_\epsilon}  L.
	\end{equation*}
	Then letting $\epsilon \rightarrow 0$, we have that 
	$\phi(\lambda_0,\eta_0) \leq L(\lambda_0)$,
	which gives the desired inequality.
	
	Now let us prove the opposite inequality.
	Assume for now that
	 $\eta_0 \neq 0$.
	Let $r_0 = |\eta_0|$, and let us first notice that
	\begin{equation}\label{eq_limsup}
	\mbox{ for
		almost every } t \in S^1,
	\quad \limsup_j L_{n_j}(\lambda_0, r_0 e^{it})=L(\lambda_0). 
	\end{equation}
	Indeed, for
	any $t \in S^1$ we have
	\begin{equation}\label{eq_limsup2}
	\lim_j L_{n_j}(\lambda_0, r_0 e^{it}) \leq \limsup_j L_{n_j}^+(\lambda_0, r_0 e^{it}) = L(\lambda_0)
	\end{equation}
	and by Fatou's lemma (applied to the functions 
	$t \mapsto -L_{n_j}(\lambda_0, r_0 e^{it})$, which are bounded from below by a constant) and the pointwise
	convergence of $L_n^{r_0}$ we get:
	\[
	\begin{aligned}
	L(\lambda_0) & =   \lim_n L_n^{r_0}(\lambda_0) = 
	\limsup_j \frac{1}{2\pi} \int_{0}^{2\pi} L_{n_j}(\lambda_0, r_0 e^{it})  dt\\
	&	 \leq \frac{1}{2\pi} \int_{0}^{2\pi} \limsup_j L_{n_j}(\lambda_0, r_0 e^{it})  dt 
	\end{aligned}
	\]
	which, together with \eqref{eq_limsup2}, 
	concludes the proof of
	\eqref{eq_limsup}.

	Suppose now to obtain a contradiction that $\phi(\lambda_0,\eta_0) < L(\lambda_0)$.
	Since $L$ is continuous and $\phi$ is upper semi-continuous, there is $\epsilon>0$
	and a neighbourhood $V_0$ of $(\lambda_0,\eta_0)$ such that 
	$\phi(\lambda,\eta) - L(\lambda) < - \epsilon$
	for all $(\lambda,\eta) \in V_0$,
	We may assume without loss of generality that $V_0 = B_0 \times \D(\eta_0,\gamma)$, where
	$B_0$ is a ball containing $\lambda_0$. Hartogs' Lemma then gives
	\begin{equation*}
	\limsup_j  \sup_{V_0} L_{n_j}-L \leq \sup_{V_0} \phi - L \leq - \epsilon.
	\end{equation*}
	But this contradicts \eqref{eq_limsup}.

	Therefore, we have proved that any convergent subsequence of $L_n$ in the $\lloc$ topology of
	$M \times \C$
	must agree with $L$ on $M \times \C^*$. Since $M \times \{0\}$ is negligible, 
	this proves that $L_n$ converges in $\lloc$ to $L$ on $M \times \C$. The proof is complete.
\end{proof}

\subsection{Equidistributions for polynomial skew products}\label{section_equidistribution_skew}

We now explain how to adapt (the proof of)
the general Theorem \ref{th:convproduit}
 to get a natural
  equidistribution statement
for a family $(f_\lambda)_{\lambda \in M}$ of polynomial skew products of $\P^2$.
Since the construction is very similar to the one above, we will omit part of the proofs.
Recall that $Q_{z,\lambda}^n$ is defined by \eqref{eq_for_qn}.

\begin{prop}\label{prop:defpotentielskew}
Let $(f_\lambda)_{\lambda \in M}$ be a holomorphic family of 
polynomial skew products
 of $\ptwo$ 
over a fixed base $p$.
	There exists a sequence of holomorphic maps $P_n^v : M \times \C \rightarrow \C$ such that:
	\begin{enumerate}
		\item For all $\lambda \in M$, $P_n^v(\lambda, \cdot)$ is a monic polynomial 
		\item If $\eta \neq 1$, then $P_n^v(\lambda,\eta)=0$ if and only if 
		there exists $(z,w) \in \C^2$ that is periodic of exact period $n$, and $(Q_z^n)'(w)=\eta$;
		\item If $\eta=1$, then $P_n^v(\lambda,\eta)=0$ if and only if there exists $(z,w) \in \C^2$
		such that $(z,w)$ is periodic of exact period $m$ dividing $n$ for $f_\lambda$, 
		and $(Q_{z,\lambda}^m)'(w)$ is a primitive $\frac{n}{m}$-th root of unity.
	\end{enumerate}
\end{prop}

\begin{proof}
	For any $n \in \N$, let $E_n(p)$ denote the set of periodic points for $p$ of exact period $n$. Let
	$$P_n^v(\lambda,\eta):=\prod_{m|n} \prod_{z \in E_m(p)} P_{z, \frac{n}{m}}(\lambda,\eta)$$
	where $P_{z,\frac{n}{m}}: M \times \C \to \C$ is the map given by \cite[Theorem 2.1]{bassanelli2011distribution}
	for the family of degree $d^m$ polynomials $\{ Q_{z,\lambda}^m : \lambda \in M \}$ with 
	$k:=\frac{n}{m}$. It is straightforward to check that $P_n$ satisfies the required properties.
\end{proof}

\begin{defi}
	For any $\eta \in \C$, we set 
	$\per_n^v(\eta):=\{ \lambda \in M: P_n^v(\lambda,\eta)=0  \}.$
\end{defi}

\begin{teo}\label{th:equidistribskew}
Let $(f_\lambda)_{\lambda \in M}$ be
 a holomorphic family of 
polynomial skew products
 of $\ptwo$ of degree $d\geq 2$
over a fixed base $p$.	
	Assume there exists $\lam_0\in M$ such that $f_{\lam_0}$ is Axiom A.
	For all $\eta \in \C$ outside of a polar subset, we have
	$d^{-kn} [\per_n^v(\eta)] \to \tbif.$
\end{teo}

\begin{proof}
	Set 
	$L_n^v(\lambda,\eta):=
(d^{nk})^{-1}	
	\log \left|P_n^v(\lambda,\eta)\right|.$
	Similarly to the proof of 
	Theorem 
	\ref{th:convproduit},
	in order to prove Theorem \ref{th:equidistribskew}, it is enough to prove 
	that $L_n^v \to L_v$ in $\lloc(M \times \C)$.
	Indeed, in the family $(f_\lambda)_{\lambda \in M}$, the exponent $L_p$
	is constant
	so $\tbif = dd^c L_v$ (see \eqref{eq_lyapunov}).
	Set
	\begin{enumerate}
		\item $L_n^{v,+} (\lambda,\eta) = (nd^{nk})^{-1} \sum_{(z,w) \in E_n(\lambda)} \log^+ |\eta -
		 (Q_{z,\lambda})'(w)|$;
		\item $L_n^{v,r}(\lambda) = (2\pi  d^{nk})^{-1} \int_{0}^{2\pi} \log |P_n^v(\lambda, r e^{it})| dt$.
	\end{enumerate}
	
	The desired convergence of $L_n^v$ follows from the convergences of $L_n^{v,+}$ and $L_n^{v,r}$ to $L^v$.
	The proof of these last two, as well as the deduction of the convergence of $L_n^v$,
	is an adaptation of the methods of the previous section, see Lemmas \ref{lem:L+bded}, \ref{lemma_conv_lr}
	and Theorem  \ref{th:convproduit}
	respectively. Note that if $f_{\lambda_0}$ is an Axiom A skew product, then vertical eigenvalues
	must converge either to $0$ or $\infty$ exponentially fast with respect to the period; therefore in 
	the annulus $\{\frac{1}{2}< |\eta|< 2  \}$, there 
	can be only finitely many vertical eigenvalues for $f_{\lambda_0}$, so that part of the assumption
	in Theorem \ref{th:convproduit} 
	is automatically satisfied.
\end{proof}

\bibliographystyle{alpha}
\bibliography{biblio}

\newcommand{\etalchar}[1]{$^{#1}$}
\begin{thebibliography}{ABD{\etalchar{+}}16}

\bibitem[ABD{\etalchar{+}}16]{astorg2014two}
Matthieu Astorg, Xavier Buff, Romain Dujardin, Han Peters, and Jasmin Raissy.
\newblock A two-dimensional polynomial mapping with a wandering {F}atou
  component.
\newblock {\em Ann. of Math.}, 184:263--313, 2016.

\bibitem[BB09]{bassanelli2009lyapunov}
Giovanni Bassanelli and Fran{\c{c}}ois Berteloot.
\newblock Lyapunov exponents, bifurcation currents and laminations in
  bifurcation loci.
\newblock {\em Mathematische Annalen}, 345(1):1--23, 2009.

\bibitem[BB11]{bassanelli2011distribution}
Giovanni Bassanelli and Fran{\c{c}}ois Berteloot.
\newblock Distribution of polynomials with cycles of a given multiplier.
\newblock {\em Nagoya Mathematical Journal}, 201:23--43, 2011.

\bibitem[BB18a]{bb_hausdorff}
Fran{\c{c}}ois Berteloot and Fabrizio Bianchi.
\newblock Perturbations d'exemples de {L}att\`es et dimension de {H}ausdorff du
  lieu de bifurcation.
\newblock {\em Journal de Mathématiques Pures et Appliquées}, 116:161--173,
  2018.

\bibitem[BB18b]{bb_warsaw}
Fran{\c{c}}ois Berteloot and Fabrizio Bianchi.
\newblock Stability and bifurcations in projective holomorphic dynamics.
\newblock {\em Banach Center Publications}, 115:37--71, 2018.

\bibitem[BBD18]{bbd2015}
Fran{\c{c}}ois Berteloot, Fabrizio Bianchi, and Christophe Dupont.
\newblock Dynamical stability and {L}yapunov exponents for holomorphic
  endomorphisms of {$\mathbb{P}^k$}.
\newblock {\em Annales scientifiques de l'ENS}, 51(1):215--262, 2018.

\bibitem[BDM08]{berteloot2008normalization}
Fran{\c{c}}ois Berteloot, Christophe Dupont, and Laura Molino.
\newblock Normalization of bundle holomorphic contractions and applications to
  dynamics (normalisation de contractions holomorphes fibr{\'e}es et
  applications en dynamique).
\newblock {\em Annales de l'institut Fourier}, 58(6):2137--2168, 2008.

\bibitem[Ber13]{berteloot2013bifurcation}
Fran{\c{c}}ois Berteloot.
\newblock Bifurcation currents in holomorphic families of rational maps.
\newblock In {\em Pluripotential theory}, pages 1--93. Springer, 2013.

\bibitem[BG15a]{berteloot2015geometry}
Fran{\c{c}}ois Berteloot and Thomas Gauthier.
\newblock On the geometry of bifurcation currents for quadratic rational maps.
\newblock {\em Ergodic Theory and Dynamical Systems}, 35(5):1369--1379, 2015.

\bibitem[BG15b]{buff2015quadratic}
Xavier Buff and Thomas Gauthier.
\newblock Quadratic polynomials, multipliers and equidistribution.
\newblock {\em Proceedings of the American Mathematical Society},
  143(7):3011--3017, 2015.

\bibitem[BH]{buffhubbard}
Xavier Buff and John~Hamal Hubbard.
\newblock {\em Dynamics in One Complex Variable}.
\newblock To be published by Matrix Edition, Ithaca, NY.

\bibitem[Bia19]{b_misiurewicz}
Fabrizio Bianchi.
\newblock Misiurewicz parameters and dynamical stability of polynomial-like
  maps of large topological degree.
\newblock {\em Mathematische Annalen}, 373(3-4):901--928, 2019.

\bibitem[Bie19]{biebler2019lattes}
Sébastien Biebler.
\newblock Lattès maps and the interior of the bifurcation locus.
\newblock {\em Journal of Modern Dynamics}, 15:95--130, 2019.

\bibitem[BT17]{bt_desboves}
Fabrizio Bianchi and Johan Taflin.
\newblock Bifurcations in the elementary {D}esboves family.
\newblock {\em Proceedings of the AMS}, 145(10):4337--4343, 2017.

\bibitem[Dem97]{demailly1997complex}
Jean-Pierre Demailly.
\newblock {\em Complex analytic and differential geometry}.
\newblock 1997.

\bibitem[DeM01]{demarco2001dynamics}
Laura DeMarco.
\newblock Dynamics of rational maps: a current on the bifurcation locus.
\newblock {\em Mathematical Research Letters}, 8(1-2):57--66, 2001.

\bibitem[DF08]{dujardin2008distribution}
Romain Dujardin and Charles Favre.
\newblock Distribution of rational maps with a preperiodic critical point.
\newblock {\em American Journal of Mathematics}, 130(4):979--1032, 2008.

\bibitem[DH84]{notesorsay2}
A.~Douady and J.~H. Hubbard.
\newblock {\em \'Etude dynamique des polyn\^omes complexes. {P}artie {I}},
  volume~84 of {\em Publications Math\'ematiques d'Orsay [Mathematical
  Publications of Orsay]}.
\newblock Universit\'e de Paris-Sud, D\'epartement de Math\'ematiques, Orsay,
  1984.

\bibitem[DH08]{demarco2008axiom}
Laura DeMarco and Suzanne~Lynch Hruska.
\newblock Axiom a polynomial skew products of $\mathbb{C}^2$ and their
  postcritical sets.
\newblock {\em Ergodic Theory and Dynamical Systems}, 28(6):1749--1779, 2008.

\bibitem[DS06]{dinh2006geometry}
Tien-Cuong Dinh and Nessim Sibony.
\newblock Geometry of currents, intersection theory and dynamics of
  horizontal-like maps.
\newblock {\em Annales de l'institut Fourier}, 56(2):423--457, 2006.

\bibitem[DS10]{ds_cime}
Tien-Cuong Dinh and Nessim Sibony.
\newblock Dynamics in several complex variables: endomorphisms of projective
  spaces and polynomial-like mappings.
\newblock In {\em Holomorphic dynamical systems}, pages 165--294. Springer,
  2010.

\bibitem[DT18]{dupont2018dynamics}
Christophe Dupont and Johan Taflin.
\newblock Dynamics of fibered endomorphisms of $\mathbb{P}^k$.
\newblock {\em arXiv preprint arXiv:1811.06909}, 2018.

\bibitem[Duj04]{dujardin2004henon}
Romain Dujardin.
\newblock H{\'e}non-like mappings in $\mathbb{C}^2$.
\newblock {\em American Journal of Mathematics}, 126(2):439--472, 2004.

\bibitem[Duj07]{dujardin2007continuity}
Romain Dujardin.
\newblock Continuity of lyapunov exponents for polynomial automorphisms of
  $\mathbb{C}^2$.
\newblock {\em Ergodic Theory and Dynamical Systems}, 27(4):1111--1133, 2007.

\bibitem[Duj11]{dujardin2011bifurcation}
Romain Dujardin.
\newblock Bifurcation currents and equidistribution on parameter space.
\newblock {\em Frontiers in Complex Dynamics}, 2011.

\bibitem[Duj16]{dujardin2016nonlaminar}
Romain Dujardin.
\newblock A non-laminar dynamical green current.
\newblock {\em Mathematische Annalen}, 365(1-2):77--91, 2016.

\bibitem[Duj17]{dujardin2016non}
Romain Dujardin.
\newblock Non density of stability for holomorphic mappings on $\mathbb{P}^k$.
\newblock {\em Journal de l'{\'E}cole polytechnique}, 4:813--843, 2017.

\bibitem[FG15]{favre2015distribution}
Charles Favre and Thomas Gauthier.
\newblock Distribution of postcritically finite polynomials.
\newblock {\em Israel Journal of Mathematics}, 209(1):235--292, 2015.

\bibitem[Gau16]{gauthier2016equidistribution}
Thomas Gauthier.
\newblock Equidistribution towards the bifurcation current i: Multipliers and
  degree d polynomials.
\newblock {\em Mathematische Annalen}, 366(1-2):1--30, 2016.

\bibitem[GOV17]{gauthier2017hyperbolic}
Thomas Gauthier, Y{\^u}suke Okuyama, and Gabriel Vigny.
\newblock Hyperbolic components of rational maps: {Q}uantitative
  equidistribution and counting.
\newblock {\em arXiv preprint arXiv:1705.05276}, 2017.

\bibitem[H{\"o}r07]{hormander2007notions}
Lars H{\"o}rmander.
\newblock {\em Notions of convexity}.
\newblock Springer Science \& Business Media, 2007.

\bibitem[Jon99]{jonsson1999dynamics}
Mattias Jonsson.
\newblock Dynamics of polynomial skew products on $\mathbb{C}^2$.
\newblock {\em Mathematische Annalen}, 314(3):403--447, 1999.

\bibitem[Lev82]{levin1982bifurcation}
Genadi~M Levin.
\newblock Bifurcation set of parameters of a family of quadratic mappings.
\newblock {\em Approximate methods for investigating differential equations and
  their applications}, pages 103--109, 1982.

\bibitem[Lev90]{levin1990theory}
Genadi~M Levin.
\newblock Theory of iterations of polynomial families in the complex plane.
\newblock {\em Journal of Mathematical Sciences}, 52(6):3512--3522, 1990.

\bibitem[Lyu83]{lyubich1983some}
Mikhail Lyubich.
\newblock Some typical properties of the dynamics of rational maps.
\newblock {\em Russian Mathematical Surveys}, 38(5):154--155, 1983.

\bibitem[MSS83]{mane1983dynamics}
Ricardo Man{\'e}, Paulo Sad, and Dennis Sullivan.
\newblock On the dynamics of rational maps.
\newblock {\em Annales scientifiques de l'{\'E}cole Normale Sup{\'e}rieure},
  16(2):193--217, 1983.

\bibitem[Oku14]{okuyama2014equidistribution}
Y{\^u}suke Okuyama.
\newblock Equidistribution of rational functions having a superattracting
  periodic point towards the activity current and the bifurcation current.
\newblock {\em Conformal Geometry and Dynamics of the American Mathematical
  Society}, 18(12):217--228, 2014.

\bibitem[Pha05]{pham}
Ngoc-Mai Pham.
\newblock Lyapunov exponents and bifurcation current for polynomial-like maps.
\newblock {\em ArXiv preprint math/0512557}, 2005.

\bibitem[Ran95]{ransford1995potential}
Thomas Ransford.
\newblock {\em Potential theory in the complex plane}, volume~28.
\newblock Cambridge university press, 1995.

\bibitem[Taf17]{taflin_blender}
Johan Taflin.
\newblock Blenders near polynomial product maps of {$\mathbb{C}^2$}.
\newblock {\em Preprint arXiv:1702.02115}, 2017.

\end{thebibliography}

\end{document}